\documentclass[reqno,11pt]{amsart}
\usepackage{amssymb}
\usepackage{amsmath,amsthm,amscd}

\newtheorem{proposition}{Proposition}[section]
\newtheorem{definition}[proposition]{Definition}
\newtheorem{lemma}[proposition]{Lemma}
\newtheorem{theorem}[proposition]{Theorem}
\newtheorem{corollary}[proposition]{Corollary}
\newtheorem{conjecture}[proposition]{Conjecture}

\newtheorem{exercise}[proposition]{Exercise}
\newtheorem{problem}[proposition]{Problem}
\newtheorem{fact}[proposition]{Fact}

\newcommand \clos{\operatorname{clos}}

\newcommand \fourier{\widehat}
\newcommand \la{\langle}
\newcommand \ra{\rangle}
\newcommand \valg{\Val^G}
\newcommand \valun{\Val^{U(n)}}

\newcommand \valson{\Val^{SO(n)}}

\newcommand\V{\mathcal V}
\newcommand \tildiota{\widetilde\iota}

\newcommand \tildk{\widetilde k}
\newcommand\CPn{\C P^n}

\newcommand \kson{k_{SO(n)}}
\newcommand \ason{a_{SO(n)}}

\newcommand \valsm{\Val^{sm}}

\newcommand \sltwo{\mathfrak{sl}(2)}
\newcommand\Hom{\mathbf{Hom}}
\newcommand\length{\operatorname{length}}
\newcommand\graph{\operatorname{graph}}
\newcommand\curv{\operatorname{Curv}}
\newcommand\bdry{\partial}
\newcommand \lcur{[\![} 
\newcommand\rcur{]\!]}

\newcommand \restrict{\llcorner} 

\newcommand\overpsi{\overline\Psi}

\DeclareMathOperator{\vol}{vol}
\DeclareMathOperator{\dist}{dist}

\newcommand{\re}{\operatorname{Re}}
\newcommand{\norm}[1]{\left\| #1 \right\|}
\newcommand{\R}{\mathbb{R}}
\newcommand{\Rn}{\mathbb{R}^n}
\newcommand{\HH}{\mathbb{H}}

\newcommand{\C}{\mathbb{C}}
\newcommand{\Cn}{\mathbb{C}^n}
\newcommand{\K}{\mathcal{K}}

\newcommand{\Ksm}{\mathcal{K}^{sm}}

\newcommand{\CP}{\mathbb{C} P}

\newcommand{\barson}{\overline{SO(n)}}
\newcommand{\barun}{\overline{U(n)}}
\newcommand{\barg}{\overline{G}}
\newcommand{\barGr}{\overline{\Gr}}

\newcommand{\vgm}{\V^{G}(M)}
\newcommand{\vgmstar}{\V^{G *}(M)}

\newcommand\VLun{\V^{\C,\lambda,n}}
\newcommand\Voneun{\V^{\C,1,n}}

\newcommand\two {I\! I}

\newcommand\piinv{\pi^{-1}}

\newcommand{\eps}{\epsilon}

\DeclareMathOperator{\kl}{Kl}
\DeclareMathOperator{\Val}{Val}
\DeclareMathOperator\Pf{Pf}
\DeclareMathOperator{\Gr}{Gr}
\DeclareMathOperator{\AGr}{\overline{ Gr}}

\DeclareMathOperator{\Id}{\operatorname{Id}}

 \title{Algebraic integral geometry}
 \author{Joseph H.G. Fu}
  \email{fu@math.uga.edu}
  \date{\today}
 \address{ Department of Mathematics, 
University of Georgia, 
Athens, GA 30602, USA}

\thanks{Partially supported
  by NSF grant DMS-1007580}
\begin{document}

\maketitle

\tableofcontents
\setcounter{tocdepth}{200}

\section{Introduction} Recent work of S. Alesker has catalyzed a flurry of progress in Blaschkean integral geometry, and opened the prospect of further advances. 
By this term we understand the circle of ideas surrounding the {\it kinematic formulas} \eqref{bkfs} below, which express certain fundamental integrals relating to the intersections of  two subspaces $K,L \subset \R^n$ in general position in terms of certain ``total curvatures" of $K$ and $L$ separately. The purest form of this study concerns the integral geometry of {\it isotropic spaces} $(M,G)$, i.e. Riemannian manifolds $M$ equipped with a subgroup $G$ of the full isometry group that acts transitively on the tangent sphere bundle $SM$. 
That kinematic formulas exist in this generality was first established in \cite{fu90}. However it was not until the discovery by Alesker of the rich algebraic structure on the space of convex valuations, and his extension of these notions to arbitrary smooth manifolds, that the determination of the actual formulas has become feasible. The key fact is the Fundamental Theorem of Algebraic Integral Geometry, or {\it ftaig}  (Thms. \ref{ftaig} and \ref {compact ftaig} below), which relates the kinematic formulas for an isotropic space $(M,G)$ to the finite-dimensional algebra of $G$-invariant valuations on $M$.

The classical approach to integral geometry relies on the explicit calculation of multiple integrals (the {\it template method} described in Section \ref{classical kfs section} below)  and intricate numerical identities. Even apart from the unsatisfyingly ad hoc nature of these arguments, they are also notoriously difficult in practice: for example, even such an eminently skilled calculator as S.-S. Chern \cite{chern66} published fallacious values for some basic integral geometric constants (cf. \cite{nijenhuis}). While Alesker theory has not entirely expunged such considerations from integral geometry, it has provided an array of ``soft" tools that give a structural rationale and thereby a practical method of checking. As a result many old mysteries have been resolved and more new ones uncovered. In a general way this development was foreseen by G.-C. Rota \cite{klain-rota}, who famously characterized integral geometry as ``continuous combinatorics".

\subsection{About these notes} These notes represent  a revision of the notes for my Advanced Course on Integral Geometry and Valuation Theory at the Centre de Recerca Matem\`atica, Barcelona in September 2010. 
  My general aim was to convey an intuitive working knowledge of the state of the art in integral geometry in the wake of the Alesker revolution. As such I have tried to include exact formulas wherever possible, without getting bogged down in technical details in the proofs and formal statements. As a result some of the proofs and statements are incomplete or informal, e.g. sometimes concluding the discussion once I feel the main points have been made or  when I don't see how to bring it to an end in a quick and decisive fashion. Some of the unproved assertions here have been treated more fully in the recent paper \cite{bfs}.
   I intend to continue to develop this account to the point where it is  reliable and complete, so comments on any aspect (especially corrections) either in person or by email will be welcome.

 Of course the topics included reflect my personal understanding and limitations. For example, there is nothing here about Bernig's determination of the integral geometry of
 $$(\C^n,SU(n)), (\R^8, Spin(7)), (\R^7,G_2).$$
 Bernig's excellent account \cite{bernigsurvey} discusses these topics and offers a different viewpoint on the theory as a whole. 
 
%
%

There are (at least) two other senses covered by the term ``integral geometry" that we will not discuss here. First is the integral geometry initiated by Gel'fand, having to do with the Radon transform and its cousins. Remarkably, recent work of Alesker suggests that the concept of valuation may prove to be the correct formal setting also for this ostensibly separate study. Second is a raft of analytic questions which may be summarized as: which spaces and subspaces are subject to the kinematic formulas? They are known to include disparate classes coming from algebraic geometry, hard analysis and convexity, and seem all to share certain definite aspects of a riemannian space \cite{fu00, fu04}.

The Exercises include assertions that I have not completely thought through, so they may be much harder or much easier than I think, and also may be false. Some statements I have labeled as Problems, which means that I believe they represent serious research projects.

\subsection{Acknowledgments} I am extremely grateful to the Centre de Recerca Matem\`atica, and to the organizers E. Gallego, X. Gual, G. Solanes and E. Teufel, for their invitation to lecture on this subject in the Advanced Course on Integral Geometry and Valuation Theory, which provided the occasion for these notes. In addition I would like to thank A. Bernig and S. Alesker for extremely fruitful collaborations and discussions over the past few years as this material was worked out.

%

%

\subsection{Notation and conventions}
Classic texts on the subject of integral geometry include 
\cite{santalo}, \cite{schneider}, \cite{klain-rota}, \cite{gray}, \cite{bonnesen-fenchel}. 

The volumes the of the $k$-dimensional unit ball and the $k$-dimensional unit sphere are denoted respectively by
\begin{align*}
\omega_k&:= \frac{\pi^{\frac k 2}}{\Gamma(1 +\frac k 2)} \\
\alpha_k&:= (k+1)\, \omega_{k+1} .
\end{align*}

$\Gr_k(V)$ denotes the Grassmannian of all $k$-dimensional vector subspaces of the real vector space $V$. $\barGr_k(V)$ denotes the corresponding space of {\it affine} subspaces.

Given a group $G$ acting linearly on a vector space $V$, we denote by $\barg:= G\ltimes V$ the associated group of affine transformations. 

If $G$ acts transitively by isometries on a riemannian manifold $M$ we will usually normalize the Haar measure $dg$ on $G$ so that 
\begin{equation}\label{eq:measure convention}
dg(\{g: go \in S\}) = \vol(S)
\end{equation}
for $S\subset M$, where $o \in M$ is an arbitrarily chosen base point.

If $S$ is a subset of a vector space then $\la S\ra$ denotes its linear (or sometimes affine) span.

Sometimes $|S|$ will denote the volume of $S$, with the dimension understood from the context.
Sometimes $\pi_E$ will denote orthogonal projection onto $E$, and others it will be projection onto the factor $E$ of a cartesian product.

For simplicity I will generally endow all vector spaces with euclidean structures and all manifolds with riemannian metrics. This device carries the advantage that the dual spaces and cotangent vectors can then be identified with the original spaces and tangent vectors. Unfortunately it also obscures the natural generality of some of the constructions, and can lead to outright misconceptions (e.g. in applications to Finsler geometry) if applied with insufficient care. 

\section{Classical integral geometry} We consider $\Rn$ with its usual euclidean structure, together with the group $\barson:= SO(n)\ltimes \Rn$ of orientation-preserving isometries. Let $\K= \K(\Rn)$ denote the metric space of all compact convex subsets $A\subset \Rn$, endowed with the Hausdorff metric
\begin{equation}
d(A,B):= \inf\{r\ge 0: A\subset B_r, \ B\subset A_r\}
\end{equation}
where $B_r$ is the ball of radius $r$ and
$$
A_r:= \{x \in \Rn: \dist (x,A) \le r\}.
$$
We denote by $\Ksm= \Ksm(\Rn)$ the dense subspace of convex subsets $A$ with nonempty interior and smooth boundary, and such that all principal curvatures $k_1,\dots k_{n-1} >0$ at every point of $\bdry A$.

\subsection{Intrinsic volumes and Federer curvature measures} Our starting point is 

\begin{theorem}[Steiner's formula]\label{steiner's formula}
If $A\in \K$ then
\begin{equation}\label{steiner}
\vol(A_r) = \sum_{i=0}^n \omega_{n-i}\, \mu_{i}(A) \,r^{n-i}, \quad r\ge 0.
\end{equation}
\end{theorem}
The functionals $\mu_i$ (thus defined) are called the {\bf intrinsic volumes}, and are equal up to scale to the {\it Quermassintegrale} introduced by Minkowski (cf. \cite{bonnesen-fenchel}, p. 49). A simple geometric argument shows that 
\begin{equation}\label{eq:vol is vol}
A\in \K, \dim A = j \implies \mu_j(A) = \vol_j(A).
\end{equation}

It is easy to prove \eqref{steiner} if $A\in \Ksm$. In this case 
\begin{equation}
A_r = A\cup \exp_A (\bdry A \times [0,r])
\end{equation}
where
\begin{equation}
\exp_A(x,t): = x+ tn_A(x)
\end{equation}
and $n_A: \bdry A\to S^{n-1}$ is the Gauss map. It is clear that $\exp_A$ gives a diffeomorphism $\bdry A \times (0,\infty) \to \Rn-A$: in fact the inverse map may be written in terms of
$$
x = \pi_A(\exp_A(x,t)), \quad t = \delta_A (\exp_A(x,t))
$$
where $\pi_A:\Rn\to A$ is the nearest point projection and $\delta_A$ is the distance from $A$. Clearly
\begin{equation}
D_{x,t_0}\exp_A(v + c \partial_t) = v + t_0L_x(v) + cn_A(x)
\end{equation}
for $v \in T_x \bdry A$, where $L_x:T_x\bdry A \to T_x\bdry A= T_{n_A(x)}S^{n-1} \subset \Rn $ is the Weingarten map. Thus the area formula gives
\begin{align*}
\vol(A_r) &= \vol(A) + \int _{\bdry A} \, d\vol_{n-1}x \int _0^r \, dt \det (\Id_{T_x\bdry A} + tL_x) \\
& = \vol(A) + \int _{\bdry A} \, d\vol_{n-1}x \int _0^r \, dt \prod_{j=1}^{n-1} (1 + tk_j)\\
& = \vol(A) + \sum_{j=0}^{n-1} \frac {r^{j+1}} {j+1}\int _{\bdry A}\sigma_{j}(k_1,\dots,k_{n-1}) \, d\vol_{n-1}x 
\end{align*}
where the $k_j$ are the principal curvatures and  $\sigma_j$ is the $j$th elementary symmetric polynomial. In particular
\begin{equation}\label{smooth mu}
\mu_i(A) = \frac 1 {(n-i) \omega_{n-i}} \int_{\bdry A} \sigma_{n-i-1}(k_1,\dots, k_{n-1})
\end{equation}
in this case. Note that the $\mu_i$ are independent of the ambient dimension, i.e. if $j:\Rn\to \R^N$ is a linear isometry then $\mu_i(j(A)) = \mu_i(A)$. Observe that
\begin{align*}
\mu_n(A)&=|A| \\
\mu_0(A)& =1 \quad\text{(by Gauss-Bonnet)} \\
\mu_i(t A)&=t^i\mu_i(A) \quad \text{ for } 0\ne t \in \R \\
\mu_i(x+A) &= \mu_i(A) \quad \text{ for } x \in \Rn .\\
\end{align*}

A modification of this approach establishes \eqref{steiner} for general $A \in \K$.  Note that the function $\delta_A:= \dist(\cdot, A)$ is $C^1$ when restricted to the complement of $A$, with 
\begin{equation}
\nabla \delta_A (x) =\frac{x-\pi_A(x)}{|x-\pi_A(x)|}.
\end{equation}
Since $\pi_A$ is clearly Lipschitz (with constant 1), in fact $\delta_A \in C^{1,1}_{loc}(\Rn - A)$, and the map $\Pi_A: \Rn-A \to S \Rn := \Rn\times S^{n-1}$ given by
\begin{equation}
\Pi_A(x) := (\pi_A(x), \nabla \delta_A (x))
\end{equation}
is locally Lipschitz. In fact, for fixed $r >0$ the restriction of $\Pi_A$ to $\bdry A_r$ is a biLipschitz homeomorphism to its image. Furthermore this image is 
independent of $r$, and these maps commute with the obvious projections between the various $\bdry A_r$. Precisely, the image is the {\bf normal cycle} of $A$,
\begin{equation}\label{eq:def N}
N(A):= \{(x,v) \in S\Rn: \langle v, x-y \rangle \ge 0 \text{ for all } y \in A\},
\end{equation}
which thus may be regarded as an oriented Lipschitz submanifold of $S\Rn$ of dimension $n-1$. Now the exponential map $\exp:S\Rn\times \R\to \Rn, \ \exp(x,v;t):= x+tv$ yields a locally biLipschitz homeomorphism $N(A) \times (0,\infty) \to \Rn-A$, with $A_r-A = \exp(N(A) \times (0,r])$. Thus the volume of $A_r$ may be expressed as
\begin{equation}\label{tube}
\vol(A_r) = \vol(A) + \int_{N(A) \times (0,r]} \exp^*d\vol
\end{equation}
where 
\begin{equation}
\exp^*(d\vol) = d(x_1 + t v_1)\wedge\dots\wedge d(x_n+ tv_n) 
\end{equation}
is a differential form of degree $n$ on $S\R^n$. Since $d\vol$ is invariant and $\exp$ is covariant under the action of $\barson$, this form is again $\barson$-invariant.

In order to understand this form it is therefore enough to evaluate it at the special point $(0,e_n) \in S\Rn$, where 
$$
v_n = 1, \quad v_1=\dots = v_{n-1} = 0, \quad dv_n = 0, \quad dx_n = \alpha
$$
where $\alpha := \sum_{i=1}^n v_i dx_i$ is the invariant {\bf canonical 1-form} of the sphere bundle $S\Rn$. Thus
\begin{align*}
\exp^*(d\vol)_{(0,e_n)}& = d(x_1 + t v_1)\wedge\dots\wedge  d(x_{n-1} + t v_{n-1})\wedge (\alpha + dt) \\
&= \left(\sum_{i=0}^{n-1} t^i \kappa_{n-i-1}\right)\wedge (\alpha + dt)
\end{align*}
for some invariant forms $\kappa_i\in \Omega ^{n-1,\barson}(S\Rn)$ of degree $n-1$, and \eqref{tube} may be expressed as
\begin{equation}
\vol(A_r) = \vol(A) + \sum_{i=0}^{n-1} \frac{r^{i+1}}{i+1}\int_{N(A)} \kappa_{n-i-1}.
\end{equation}
Thus for general $A \in \K$
\begin{equation}\label{general mu}
\mu_i(A) = \frac 1 {(n-i) \omega_{n-i}} \int_{N(A)} \kappa_{i}.
\end{equation}
The formula \eqref{general mu} is of course a direct generalization of  \eqref{smooth mu}, and the latter may be computed directly from the former by pulling back via the diffeomorphism $\bar n_A:\bdry A \to N(A)$, $\bar n_A(x) := (x, n_A(x))$.

It is natural to express the intrinsic volumes as the ``complete integrals" of the {\bf Federer curvature measures}
\begin{align}
\Phi_i^A (E)&:= \frac 1 {(n-i) \omega_{n-i}}\int_{N(A) \cap \pi_{\Rn}^{-1}(E)} \kappa_i , \quad i = 0,\dots, n-1,\\
\Phi_n^A(E) &:= \vol (A\cap E).
\end{align}
These measures satisfy
\begin{equation}
\vol (A_r \cap \pi_A^{-1}(E) ) = \sum_{i=0}^n \omega_{n-i} \Phi_i^A(E) r^{n-i}.
\end{equation}

The discussion above goes through also in case the compact set $A$ is only {\bf semiconvex}, i.e. there is $r_0>0$ such that if $\delta_A(x) < r_0$ then there exists a unique $\pi_A(x):= p\in A$ such that $\delta_A(x) = |x-p|$ (it is necessary of course to restrict to $r < r_0$ throughout). The supremum of such $r_0$ is called the {\bf reach} of $A$ \cite{federer59}. Note that $\operatorname{reach }A = \infty$ iff $A$ is convex. In this more general setting, the Federer curvature measures $\Phi^A_i$ will generally be signed.

\subsection{Other incarnations of the normal cycle} Consider also
\begin{align}
\vec N(A)&:= \{(x,v)\in T\R^n: x \in A, \langle v,x-y \rangle \ge 0 \text{ for all } y \in A\} \\
N_1(A)&:= \vec N(A) \cap (\R^n \times B_1)\\
\vec N^*(A)&:= \{(x,\xi)\in T^* \R^n: x \in A, \langle \xi,x-y \rangle \ge 0 \text{ for all } y \in A\}\\
N^*(A)&:= \{(x,[\xi])\in S^*\R^n:(x ,\xi)\in\vec N( A)\}
\end{align}
where $S^*\R^n$ is the cosphere bundle. The last two are in some sense the correct objects to think about, since they behave naturally under linear changes of variable.


\subsection{Crofton formulas} The intrinsic volumes occur naturally in certain questions of geometric probability. The starting point is {\bf Crofton's formula}. In the convex case this may be stated
\begin{equation}
\label{crofton}
\mu_{n-1}(A) = \int_{\Gr_{n-1}} |\pi_E(A)| \, dE.
\end{equation}
Here the codimension one intrinsic volume is half of  the perimeter and $dE$ is a Haar measure chosen below. This may be proved via the area formula. Fixing $ E \in \Gr_{n-1}$ with unit normal vector $v=v_E$, and $x \in \partial A$, the Jacobian determinant of the restriction of $ D \pi_E$ to $T_x\partial A$ is 
$$
\left. \det D \pi_E\right|_{T_x\partial A} = |v\cdot n_A(x)|.
$$
Hence
\begin{align}
\int_{\Gr_{n-1}} |\pi_E(A)| \, dE & =\frac 1 2 \int_{\Gr_{n-1}} dE \int_{\partial A}  \, |v\cdot n_A(x)|\, dx\\
&=\int_{\partial A}  dx\int_{S^{n-1}}  \, |v\cdot u| \, dv \\
&= \mu_{n-1}(A)
\end{align}
where $u\in S^{n-1}$ is an arbitrarily chosen unit vector, provided the Haar measures $dE, dv$ are normalized appropriately.

Observe that
$$
\pi_E(A_r) = [\pi_E(A)]_r.
$$
Hence by Steiner's formula
\begin{align*}
 \mu_{n-2}(A_r) &=c\frac d{dr} \mu_{n-1}(A_r) \\
&=c\frac d{dr} \int_{\Gr_{n-1}} \left|[\pi_E(A)]_r\right| \, dE\\
&=c \int_{\Gr_{n-1}}\frac d{dr} \left|[\pi_E(A)]_r\right| \, dE\\
&=c \int_{\Gr_{n-1}} \mu_{n-2}( \pi_E(A_r)) \, dE\\
&=c \int_{\Gr_{n-1}}\, dE \int_{\Gr_{n-2}(E)}  \vol_{n-2} ( \pi_F(A_r))\,dF \\ 
&= \int_{\Gr_{n-2}(\Rn)} \vol_{n-2} ( \pi_F(A_r))\,dF 
\end{align*}
with conveniently normalized Haar measure $dF$ on $\Gr_{n-2}$, since $\pi_F\circ \pi_E = \pi_F$ when $F$ is a subspace of $E$. Setting $r=0$ and continuing in this way,
\begin{proposition} \label{general crofton}
If the Haar measure $dG$ is appropriately normalized then
\begin{equation}
\mu_k = \int_{\Gr_k} \vol_k(\pi_G(\cdot)) \, dG , \quad k= 0,\dots,n-1.
\end{equation}
\end{proposition}

This can also be written in the following equivalent form. Let $\AGr_{n-k}$ denote the space of {\it affine} planes of dimension $n-k$ in $\Rn$.  Then
\begin{equation}\label{affine general crofton}
\mu_k (A)= \int_{\AGr_{n-k}} \chi(A\cap \bar H) \, d\bar H, \quad k= 0,\dots,n-1
\end{equation}
where $\chi $ is the Euler characteristic, and $d\bar H$ is the Haar measure obtained by viewing $\AGr_{n-k}$ as the total space of the tautological bundle over $\Gr_k$ and taking the product of $dG$ from Prop. \ref{general crofton} with the Lebesgue measure on the fibers.

\begin{exercise} Deduce that
\begin{align*}
\mu_k (A)&=c^m_k \int_{\Gr_m} \mu_k(\pi_E(A) ) \, dE, \quad m\ge k \\
&= d^j_k \int_{\AGr_{n-j}}\mu_{k-j}(A \cap \bar H) \, d\bar H, \quad n \ge k\ge j
\end{align*}
Is it possible to choose the normalizations of the Haar measures so that $c^m_k, d^j_k\equiv 1$?
\end{exercise}

\subsection{The classical kinematic formulas}\label{classical kfs section}
It follows immediately from \eqref{general crofton} that the $\mu_k$ are {\it continuous, translation-invariant, $SO(n)$-invariant valuations}, i.e.
\begin{enumerate}
\item \label{def val}$A,B, A\cup B \in \K \implies \mu_k(A\cup B) +\mu_k(A\cap B) =\mu_k (A) + \mu_k (B)$
\item \label{trans inv}$\mu_k(x + A) = \mu_k(A) $ for all $x \in \Rn$
\item \label{cts val}$A_i \to A$ in the Hausdorff metric $\implies \mu_k(A_i) \to \mu_k(A)$
\item\label{son inv} $\mu_k(gA) = \mu_k(A) $ for all $g \in SO(n)$
\end{enumerate}

\begin{exercise} Prove these statements.
\end{exercise}

The space of functionals satisfying \eqref{def val}, \eqref{trans inv}, \eqref{cts val} is denoted by $\Val= \Val(\Rn)$. The subspace satisfying in addition \eqref{son inv} is denoted $\valson$. 
\begin{theorem}[Hadwiger \cite{hadwiger}]\label{hadwiger thm}
$$
\valson = \langle \mu_0,\dots, \mu_n\rangle
$$
\end{theorem}
Since $\mu_k(B_r) = c_k r^k$, it is clear that the $\mu_i$ are linearly independent---
in fact  given any $n+1$ distinct radii $0\le r_0<\dots < r_n$ 
 \begin{equation}\label{independent}
\det \, [\mu_i(B_{r_j})]_{ij} \ne 0.
\end{equation}
Indeed, \eqref{independent} is (up to a constant) the determinant of a Vandermonde matrix.

Hadwiger's theorem yields immediately the following key fact. Let $\barson:= SO(n) \ltimes \Rn$ denote the group of orientation-preserving euclidean motions, with Haar measure $d\bar g$. The natural normalization is the product of Lebesgue measure $\vol$ on $\Rn$ with the probability measure on $SO(n)$, i.e. we stipulate that
\begin{equation} \label{dg normalization}
d\bar g\left(\{\bar g: \bar g (0) \in E\}\right) = \vol(E)
\end{equation}
for every measurable $E \subset \Rn$.

\begin{theorem}[Blaschke kinematic formulas] \label{bkfs}There exist constants $c^{k}_{ij}$ such that
\begin{equation}\label{blaschke kf}
\int_{\barson} \mu_k(A\cap \bar g B) \, d\bar g = \sum_{i+j = n+k} c^k_{ij} \mu_i(A) \mu_j(B)
\end{equation}
for all $A,B \in \K$.
\end{theorem}

\begin{theorem}[Additive kinematic formulas]\label{akfs}
There exist constants $d^{k}_{ij}, i+j =k$, such that
\begin{equation}\label{additive kf}
\int_{SO(n)} \mu_k(A + g B) \, d g = \sum_{i+j = k} d^k_{ij} \mu_i(A) \mu_j(B)
\end{equation}
for all $A,B \in \K$.
\end{theorem}

\begin{proof}[Proof of Theorems \ref{bkfs}, \ref{akfs}] Consider first \eqref{blaschke kf}. Fixing $B$, it is clear that the left hand sides of \eqref{blaschke kf} and \eqref{additive kf} are $\barson$-invariant valuations in the variable $A\in \K$. Furthermore, if $A_i \to A$ then  $A_i \cap \bar g B \to A\cap \bar g B$ for fixed $\bar g \in \barson$ provided either $A, \bar g B$ are disjoint or the interiors of $A , \bar g B$ intersect. Thus $\mu_k(A_i \cap \bar g B) \to\mu_k( A \cap \bar g B)$ for such $\bar g$, which clearly constitute a subset of $\barson$ of full measure. Finally, all $\mu_k(A_i \cap \bar g B) \le \mu_k(B)$. Thus the dominated convergence theorem implies that the left-hand side of \eqref {blaschke kf} $\in \valson$, and therefore 
\begin{equation}\label{prototype blaschke kf}
\int_{\barson} \mu_k(A\cap \bar g B) \, d\bar g = \sum_{i=0}^n c_i(B) \mu_i(A) 
\end{equation}
for some constants $c_i(B)$.

In view of \eqref{independent}, there are $A_0,\dots, A_n \in \K$ and constants $\alpha^i_j$ such that
\begin{equation}
\sum_{j=0}^n \alpha^i_j \int_{\barson} \mu_k(A_j\cap \bar g B) \, d\bar g = c_i(B) , \quad i =0,\dots,n.
\end{equation}
Repeating the argument of the first paragraph, this time with the $A_i$ fixed and $B$ as the variable, it follows that the $c_i\in \valson$, and Hadwiger's theorem implies that the right hand side can be expanded as indicated--- simple considerations of scaling ensure that the coefficients vanish unless $i+j=n+k$.

The proof of Theorem \ref{akfs} is completely similar.
\end{proof}

These facts may be encoded by defining the {\bf kinematic} and {\bf additive kinematic operators} $\kson,\ason: \valson\to \valson \otimes \valson$ by
\begin{align*}
\kson(\mu_k) & = \sum_{i+j = n+k} c^k_{ij}\, \mu_i\otimes \mu_j\\
\ason(\mu_k) & = \sum_{i+j = k} d^k_{ij}\,  \mu_i\otimes \mu_j
\end{align*}
where the $c^k_{ij},d^k_{ij}$ are taken from \eqref{blaschke kf}, \eqref{additive kf}.
Note that if we use the invariant probability measure on $SO(n)$ then 
$$
\kson(\chi) = \ason(\vol).
$$

Observe that although the kinematic formulas involve curvature measures in general, they give rise to the {\bf first order formulas}
\begin{equation}\label{1st order}
\int_{\barson} \vol_{n-k-l}(M^k\cap \bar g N^l) \, d\bar g = c^{n-k-l}_{kl} \vol_{k}(M^k) \vol_l(N^l)
\end{equation}
for compact $C^1$ submanifolds $M^k,N^l$ (or more generally for suitably rectifiable sets of these dimensions). These are often called {\it Poincar\'e-Crofton formulas}, but we will use this different term because of the frequent appearance of these names in other parts of the theory.

The classical approach to determining the structure constants $c^k_{ij},d^k_{ij}$ is the {\bf template method}: find enough bodies $A,B$ for which the kinematic integral can be computed directly, and solve the resulting system of linear equations. This is a bit tricky in the case of Theorem \ref{bkfs}, but for Theorem \ref{akfs} it is easy. Let us rescale the $\mu_i$ by
$$
\psi_i = \mu_i(B_1)^{-1} \mu_i
$$
so that $\psi_i(B_r) = r^i$.
 Let $A=B_r, B= B_s$ be balls of radii $r,s$, and take $dg $ to be the probability Haar measure on $SO(n)$. Then
 \begin{align*}
(r+s)^k = \int_{SO(n)} \psi_k(B_r + g B_s) \, d g &= \sum_{i+j = k} d^k_{ij} \psi_i(B_r) \psi_j(B_s) \\
& = \sum_{i+j = k} d^k_{ij} r^i s^j
\end{align*}
so
$$
\ason(\psi_k) = \sum_{i+j = k} \binom k i  \psi_i\otimes \psi_j.
$$

\begin{exercise} The operators $\kson,\ason$ are coassociative, cocommutative coproducts on $\valson$. 
\end{exercise}

(Recall that the product on an algebra $A$ is a map $A \otimes A\to A$, and that the associative and commutative properties may be stated in terms of commutative diagrams involving the product map. Thus a coproduct is a map $A\to A \otimes A$, and the coassociative and cocommutative properties are those obtained by reversing all the arrows in these diagrams.)

\begin{theorem}[Nijenhuis \cite{nijenhuis}]\label{nijenhuis}
There is a basis $\theta_0,\dots, \theta_n$ for $\valson$ such that
$$
\kson (\theta_k)= \sum_{i+j = n+k}  \theta_i\otimes \theta_j, \quad\ason(\theta_k) = \sum_{i+j = k} \theta_i\otimes \theta_j,
$$
i.e. the structure constants for both coproducts are identically equal to unity.
\end{theorem}
\begin{proof}[Start of proof] Put
$\theta'_k:= \frac {\psi_k}{k!}$. Clearly the second relation holds with $\theta_i=\theta_i'$, but we will see below that the first relation does not--- there is a positive constant $c\ne 1$ in front of the right hand side. But it is easy to check that $\theta_i:= c^{\frac i n} \theta_i'$ works.
\end{proof}

Nijenhuis speculated that some algebraic interpretation of the kinematic formulas should explain Thm. \ref{nijenhuis}. We will see that this is indeed the case.

\subsection{The Weyl principle}\label{weyl section} H. Weyl discovered that the Federer curvature measures of a smoothly embedded submanifold of euclidean space are integrals of Riemannian invariants now commonly known as the {\bf Lipschitz-Killing curvatures}. 
In particular the Federer curvature measures of a smooth submanifold depend only on its intrinsic metric structure and not on the choice of embedding in euclidean space.
This is even true of manifolds with boundary. In fact the Weyl principle applies much more broadly, but we have no systematic understanding of this phenomenon.

Let $M^k\subset \Rn$ be a smooth compact submanifold, and let $e_1,\dots,e_k$ be a local orthonormal frame, with associated coframe $\theta_i$ and curvature 2-forms $\Omega_{ij}$.

\begin{theorem}[\cite{weyl}] \label{weyl principle}
 \begin{equation}
\label{weyl interior}\mu_i(M^k)= 
\begin{cases}
(\frac2 \pi)^{\frac{k-i}2}\int_M \frac 1{k!}\sum \sigma_a\Omega_{a_1a_2}\dots \Omega_{a_{2i-1}a_{2i}}\theta_{a_{2i+1}}\dots \theta_{a_k},  \\
 \quad \quad 0\le k-i \ \text{even}\\
 0 \quad \quad \text{otherwise}
\end{cases}
\end{equation}
where the sum is over all permutations $a$ of $1,\dots,k$ and $\sigma_a$ is the sign.
\end{theorem}
\begin{proof} Put $l:= n-k$.
Let $e_{k+1},\dots,e_n$ be a local orthonormal frame for the norrmal bundle of $M$. We extend the total frame  $e_1,\dots,e_n$ to a small tubular neighborhood of $M$ by taking $e_i:= e_i \circ \pi_M$. Thus $e_{k+1},\dots, e_n$ are parallel to the fibers of $\pi_M$. Then the tube of radius $r$ around $M$ may be expressed locally as  the image of $M\times B_r$ under the map
$$
\phi: M\times \R^l \to \Rn,\quad \phi(p,y) := p + \sum_{i=1}^l y_i e_{k+i}(p).
$$
Put $\theta_i$ for the coframe dual to $e_i$, with corresponding connection forms $\omega_{ij}= e_j\cdot de_i$. The structure equations are
\begin{align} 
d\theta_i &= \sum_j \omega_{ij} \theta_j\\
d\omega_{ij}&= \sum_r \omega_{ir}\omega_{rj}= - \sum_r \omega_{ir}\omega_{jr}
\end{align}
and if $i,j\le k$
\begin{equation}
d\omega_{ij} =\Omega_{ij} - \sum_{s\le k} \omega_{is}\omega_{js}
\end{equation}
where $\Omega_{ij}$ are the curvature forms. Thus
\begin{equation}\label{cartan}
\Omega_{ij}= -\sum_{t> k} \omega_{it}\omega_{jt}.
\end{equation}

To compute the volume of the tube we note
\begin{equation}
\phi^* \theta_i = \begin{cases}\theta_i + \sum_j y_j \omega_{k+j,i}, & i \le k \\
					dy_{i-k}, & i >k
			\end{cases}
\end{equation}
Hence
\begin{align}
\notag \phi^* d\vol &= \phi^*(\theta_1 \dots \theta_n)\\
\label{integrand}&= (\theta_1 + \sum_j y_j \omega_{k+j,1})\dots(\theta_k+ \sum_j y_j \omega_{k+j,k})\, dy_1\dots dy_l.
\end{align}

We wish to integrate this over sets $U\times B(0,r)$ where $U \subset M$ is open. By symmetry it is clear that any term of odd degree in any of the $y_i$ will integrate to $0$. Furthermore, the even terms are given by certain products of the $\theta_i, \omega_{ij}$ with coefficients $ c(e_1,\dots,e_l)$, where
\begin{equation}\label{weyl integral}
c(e_1,\dots,e_l):=\int_{B(0,r)} y_1^{e_1}\dots y_l^{e_l} \, dy_1\dots dy_l  = \omega_lr^{l+e}\frac{(e_1-1)!!\dots(e_l-1)!!}{(l+2)(l+4)\dots( l+e)} 
\end{equation}
where the $e_i$ are even and $e:= \sum e_i$, using the standard trick of multiplying by the Gaussian $exp({-\sum y_j^2})$ and integrating over the space of all $y$ by iteration (cf. \cite{weyl}).

Consider the terms of degree $e$ in $\vec y$. These contribute to the coefficient of $r^{l + e}$ in the volume of the tube. For definiteness we consider the terms that are multiples of $\theta_{e+1}\dots \theta_k dy_1\dots dy_l$. The remaining factor is 
\begin{equation}\label{tube expansion}
\sum_\pi(-1)^\pi c(e_1,\dots,e_l) \left(\Pi_{i\in  \pi_1}\omega_{k+1, i}\right)\dots\left(\Pi_{i\in  \pi_l}\omega_{k+l, i}\right)
\end{equation}
over ordered partitions $\pi$  of $\{1,\dots,e\}$ into subsets of cardinality $e_1,\dots,e_l$.

 We claim that this sum is precisely $\frac{\omega_lr^{l+e}}{(l+2)(l+4)\dots(l+e)}$ times the Pfaffian 
\begin{equation}
\Pf([\Omega_{ab}]_{1\le a.b\le e})= \sum_\pi (-1)^\pi \Omega_{\pi_1,\pi_2}\dots \Omega_{\pi_{e-1},\pi_e}
\end{equation}
of the antisymmetric matrix of 2-forms $[\Omega_{ab}]_{1\le a,b\le e}$. Here $\pi$ ranges over all unordered partitions of $\{1,\dots,e\}$ into  pairs. Observe that there are precisely $(e-1)!!= \frac{e!}{2^{\frac e 2}(\frac e 2)!}$ terms in this sum.

To see this we expand each term in the Pfaffian in terms of the $\omega_{ij}$ using \eqref{cartan}.
For example,
\begin{align*}
\Omega_{12}\Omega_{34}\dots \Omega_{e-1,e} &=(-1)^{\frac e 2} \left(\sum_{t>k}\omega_{1t}\omega_{2t}\right)\dots \left(\sum_{t>k}\omega_{e-1,t}\omega_{et}\right)\\
&= (-1)^{\frac e 2} \sum_\pi \left(\Pi_{i\in  \pi_1}\omega_{k+1, i}\right)\dots\left(\Pi_{i\in  \pi_l}\omega_{k+l, i}\right)
\end{align*}
where now the sum is over  ordered partitions of $\{1,...,e\}$ into $l$ subsets that group each pair $2i-1, 2i$ together. Thus the expansion of the Pfaffian yields a sum involving precisely the same terms as in \eqref{tube expansion}, where the coefficients $c$ are replaced by $d(\pi):=$ the number of partitions, subordinate to $\pi$, into pairs. As above this is precisely $d(\pi) = (e_1-1)!! \dots (e_l-1)!!$, which establishes our claim.

The coefficient of the integral in \eqref{weyl interior} arises from the identity
$$\frac {\omega_{n-k}}{\omega_{n-i}(n-k+2)(n-k+4)\dots(n-i)} = \left(\frac 2 \pi\right)^{\frac {k-i}2}$$
for $i\le k, k-i $ even.
\end{proof}

There are also intrinsic formulas \cite{cms2} for $\mu_i(D)$ for smooth compact domains $D\subset M$, involving integrals of invariants of the second fundamental form of the boundary of $D$ relative to $M$. Thus

\begin{theorem}\label{proto LK}
If $\phi: \Rn\supset M^k \to \R^N$ is a smooth isometric embedding and $D\subset M$ is a smooth compact domain then $\mu_i(D) = \mu_i(\phi(D))$.
\end{theorem}

\section{Curvature measures and the normal cycle}
\subsection{Properties of the normal cycle} The definition \eqref{eq:def N} of the normal cycle $N(A)$ extends to any $A \in \K$, regardless of smoothness. As we have seen, it is a naturally oriented Lipschitz submanifold of dimension $n-1$, without boundary, in the sphere bundle $S\Rn$. It will sometimes be useful to think of $N(A)$ as acting directly on differential forms by integration, i.e. as a current. In these terms, Stokes' theorem implies that $N(A)$ annihilates all exact forms. 

The normal cycle is also {\bf Legendrian}, i.e. annihilates all multiples of the contact form $\alpha= \sum_{i=1}^n v_i dx_i$. This is clear for $N(A_r)$, since $\bdry A_r$ is a $C^1$ hypersurface and the normal vector annihilates the tangent spaces. It now follows for $N(A)$, since $N(A_r) \to N(A)$ as $r\downarrow 0$. Here the convergence is in the sense of the flat metric: the difference $N(A_r) -N(A)= \partial T_r$, where $T_r$ is the  $n$-dimensional Lipschitz manifold
$$
T_r:= \{(x,\nabla_x \delta_A): 0< \delta_A(x) <r\}.
$$
Clearly $\vol_n(T_r) \to 0$, which is the substance of the flat convergence. By Stokes' theorem this entails weak convergence.
 
In fact, the operator $N$ is itself a continuous current-valued valuation.

\begin{theorem} If $\K \owns A_i \to A$ in the Hausdorff metric then $N(A_i) \to N(A)$ in the flat metric.

 If $A,B,A\cup B \in \K$ then
\begin{equation}\label{N is val}
N(A\cup B) + N(A\cap B) = N(A )+N(B).
\end{equation}
\end{theorem}
\begin{proof} The first assertion follows at once from the argument above, together with the observation that for fixed $r>0$ the normal cycle $N(A_r)$ is a small (with respect to the flat metric) perturbation of $N((A_i)_r)$ for large $i$.

The second assertion is very plausible pictorially, and may admit a nice simple proof. However we will give a different sort of proof based on a larger principle. To begin, note that for generic $v\in S^{n-1}$ the halfspace $H_{v,c}:= \{x: x\cdot v \le c\}\subset\Rn $ meets $A$ iff there is a unique point $x \in A\cap H_{v,c}$ such that $(x,-v) \in N(A)$, and in fact the intersection multiplicity $(H_{v,c} \times\{-v\})\cdot N(A)$ is exactly $+ 1$. In other words
\begin{equation}
(H_{v,c} \times\{-v\})\cdot N(A) = \chi(A\cap H_{v,c}).
\end{equation}
for generic $v$ and all $c$.
This condition is enough to determine the compactly supported Legendrian cycle $N(A)$ uniquely \cite{fu90}. Thus
\begin{align*}
(H_{v,c} \times\{-v\})\cdot (N(A)+N(B)-N(A\cap B)) &= \chi(A\cap H_{v,c})+\chi(B\cap H_{v,c})\\
& \quad \quad -\chi(A\cap B\cap H_{v,c})\\
&= \chi((A\cup B)\cap H_{v,c})
\end{align*}
by the additivity of the Euler characteristic. Thus the uniqueness statement above ensures \eqref{N is val}.
\end{proof}
This argument also shows that the normal cycle of a finite union of compact convex sets set is well-defined, i.e. the inclusion-exclusion principle yields the same answer regardless of the decomposition into convex sets is chosen.
The theorem of \cite{fu90} also implies the existence and uniqueness of normal cycles for much more exotic sets, but this has to do with the analytic side of the subject, which we don't discuss here.

\label{gen curv meas}\subsection{General curvature measures} It follows from the valuation property \eqref{N is val} that 
any smooth differential form $\varphi \in \Omega^{n-1}(S\Rn)$ gives rise to a continuous valuation $\Psi_\varphi$ on $\K$. If $\varphi$ is translation-invariant then so is $\Psi_\varphi$. 

The map from differential forms to valuations factors through the space of {\bf curvature measures}, as follows. Given $A \in \K$ put $\Phi_\varphi^A$ for the signed measure
$$
\Phi_\varphi^A (E) := \int_{N(A) \cap \piinv_{\Rn} E} \varphi.
$$
Denote the space of all such assignments by $\curv(\Rn)$.
 As in the case of the Federer curvature measures (where $\varphi =\kappa_i, i =0,\dots,n-1$), if $A\in \Ksm$ then 
\begin{equation}\label{general curvature integral}
\Psi_\varphi(A) = \int_{\bdry A} P^\varphi_{x,n(x)}(\two_x^A) \, dx
\end{equation}
where for each $(x,v) \in S\Rn$ the integrand $P^\varphi_{x,v}$ is a certain type of polynomial in symmetric bilinear forms on $v^\perp$ and $\two_x^A$ is the second fundamental form of $\bdry A$ at $x$.

To be more precise about the polynomial $P$, fix a point $(x,v) \in S\Rn$ and consider the tangent space
$$
T_{x,v}S\Rn \simeq\Rn\oplus v^\perp \simeq v^\perp \oplus \langle v \rangle \oplus v^\perp =:Q\oplus \langle v \rangle
$$
where $Q= \alpha^\perp$ is the contact plane at $(x,v)$. Note that the restriction of $d\alpha$ is a symplectic form on $Q$. Now the polynomial $P$ can be characterized as follows. 
 If $A\in \Ksm$ and $(x,v)\in N(A)$ (i.e. $v=n_A(x)$), put $L: v^\perp \to v^\perp$ for the Weingarten map of $\partial A$ at $x$. Then $\graph L\subset Q$ equals the tangent space to $N(A)$ at $(x,v)$ and is a Lagrangian subspace of $Q$. This is equivalent to the well-known fact that the Weingarten map is self-adjoint.
 Put $\bar L:v^\perp \to Q$ for the graphing map $\bar L(z):= (z, Lz)$. Then the integrand of \eqref{general curvature integral} is $\bar L^* \varphi_{x,v}$ as a differential form on $\partial A$.
 

\begin{lemma}\label{curv meas ker} Let $V$ be a euclidean space of dimension $m$ and $\varphi \in \Lambda^m(V\oplus V)$. Then $\bar L^* \varphi=0$ for all self-adjoint linear maps $L:V\to V$ iff $\varphi$ is a multiple of the natural symplectic form on $V\oplus V$.
\end{lemma}
\begin{proof} The self-adjoint condition on $L$ is equivalent to the condition that the graph of $L$ be Lagrangian, so it is only necessary to prove that $\bar L^* \varphi=0\implies \varphi \in (\omega)$, where
$\omega\in \Lambda^2(V\oplus V)^*$ is the symplectic form. 
\end{proof}

\begin{proposition}\label{zero curv meas}
The curvature measure $\Phi_\varphi = 0$ iff $\varphi \in (\alpha,d\alpha)$.
\end{proposition}
\begin{proof} This follows at once from Lemma \ref{curv meas ker} and the preceding discussion.
\end{proof}

In Thm. \ref{ker thm} below we will characterize the kernel of the full map $\Psi: \Omega^{n-1}(S^*W)^W  \to \Val(W)$, due to 
Bernig and Br\"ocker \cite{bebr07}. 

\subsection{Kinematic formulas for invariant curvature measures} Let $M$ be a connected Riemannian manifold of dimension $n$ and $G$ a Lie group acting effectively and {\bf isotropically} on $M$, i.e. acting by isometries and such that the induced action on the tangent sphere bundle $SM$ is transitive. Under these conditions the group $G$ may be identified with a sub-bundle of the bundle of orthonormal frames on $M$.
The main examples are the (real) space forms $M$ with their groups $G$ of orientation-preserving isometries; the complex space forms $\C P^n$ and $\C H^n$ with their groups of holomorphic isometries, and $M=\C^n$ with $G = \barun$; and the quaternionic space forms.

Let $K\subset H\subset G$ be the subgroups fixing points $\bar o\in SM, o= \pi(\bar o)\in M$. Put $\curv^G(M)\simeq \Lambda^{n-1}(T_{\bar o} SM)^K/(\alpha,d\alpha)$ for the space of $G$-invariant curvature measures on $M$.

\begin{theorem}[\cite{fu90}]\label{kf for cm} There is a linear map
 $$\tilde k = \tilde k_G: \curv^G(M) \to \curv^G(M) \otimes \curv^G(M) 
 $$ 
 such that for any open sets $U,V \subset M$ and any sufficiently nice compact sets $A,B \subset M$
\begin{equation} 
\tilde k(\varphi) (A,U;B,V) = \int_{G} \varphi(A\cap gB, U \cap gV) \, dg.
\end{equation}
\end{theorem}

{\bf Remark.}  Sets with positive reach or subanalytic sets are all ``sufficiently nice" for this theorem to hold. Unfortunately there is no known simple characterization of the precise properties needed. The best formal result along these lines so far is Corollary 2.2.2 of \cite{fu94}.

\begin{proof} Consider the cartesian square of $G\times G$ spaces
\begin{equation}\label{main diagram}
\begin{CD}
E&@>>> &G \times SM \\
@VVV & & @VVV\\
SM \times SM &@>>>& M \times M
\end{CD}
\end{equation}
where the vertical bundles have fiber $H \times S_oM$.
Here the map on the right is $(g,\xi) \mapsto (g\pi\xi, \pi\xi )$ and the action of $G\times G$ on $G\times SM$ is $(h,k)\cdot (g,\xi):= (h gk^{-1}, k\xi)$.

The fiber of the map on the left over a point $(\eta,\zeta)$ is 
$$F_{(\eta,\zeta)}:=\{(g,\xi): g^{-1}\pi \eta = \pi \zeta = \pi \xi\}\subset G\times SM.$$
Put
\begin{equation}\label{def C}
C_{(\eta,\zeta)}:= \clos \{(g, \xi): \xi = a g^{-1}\eta + b \zeta \text{ for  some } a,b >0\}\subset F_{(\eta,\zeta)}.
\end{equation}
In other words $C_{(\eta,\zeta)}$ is a stratified space of dimension $\dim H + 1$ and consisting of all pairs $(g,\xi)$ such that $g^{-1}\eta,\zeta$ lie in a common tangent space, and $\xi$ lies on some geodesic joining these points in the sphere of this tangent space. One may check directly that $(h,k)\cdot C_{(\eta,\zeta)} = C_{(h\eta,k\zeta)}$ for $(h,k)\in G\times G$. Each $C_{(\eta,\zeta)}$ carries a natural orientation such that
$$
\partial C_{(\eta,\zeta)} = (H_{(\eta,\zeta)} \times \{\zeta\}) - \{(g,g^{-1}\eta): g \in H_{(\eta,\zeta)} \} + (K_{(\eta,\zeta)} \times S_{\pi\zeta}M)
$$
where $H_{(\eta,\zeta)}:= \{g\in G: \pi g^{-1}\eta =\pi \zeta\} $ and $H_{(\eta,\zeta)} \supset K_{(\eta,\zeta)} :=\{g: g^{-1}\eta = -\zeta\}$. For $\beta \in \Omega^*(SM)^G$ we have 
$$
\pi_{C*} (dg \wedge \beta) \in (\Omega^*(SM)\times \Omega^*(SM))^{G\times G}
$$
where $\pi_{C*}$ is fiber integration over $C$.

Now consider $C(A,B) := N(A) \times N(B) \times_E C \subset E$. One checks that the image of $C(A,B)$ in $G\times SM$ consists of all $(g,\xi)$ such that $g^{-1}\pi \eta = \pi \zeta = \pi \xi$ for some $\eta \in N(A) ,\zeta \in N(B)$ and $\xi $ lying on a geodesic between $g^{-1}\eta, \zeta$. Furthermore the set of those $g$ for which $g^{-1}\eta = -\zeta$ for some $\eta\in N(A),\zeta\in N(B)$ has positive codimension 1: it is the image in $G$ of $N(A) \times N(B) \times_E K $, which has dimension $2n-2 + \dim K$, whereas $\dim G = \dim G/K + \dim K = \dim SM + \dim K = 2n-1 + \dim K$. 

Thus $A,gB$ meet transversely for a.e. $g \in G$, and for such $g$
$$
N(A \cap gB) = N(A) \restrict \pi^{-1}(gB) + g_*N(B) \restrict \pi^{-1}A + \pi_{SM*}(C(A,B)\cap \pi_G^{-1}(g)).
$$
Now we may compute the kinematic integral for a given $\beta \in \Omega^{n-1}(SM)^G$ in either of two ways: either by pushing $N(A) \times N(B) \times_E C$ into the top right corner $G\times SM$ of \eqref{main diagram}) and integrating $dg \wedge \beta$; or else by pulling back $dg\wedge \beta$ to $E$, pushing it down to $SM \times SM$ via $\pi_{C*}$, and integrating the result over $N(A) \times N(B)$. Thus the conclusion of the theorem is fulfilled for the curvature measure $\varphi =\Phi_\beta$ with
$$
\tilde k(\Phi_\beta) = \Phi_\beta \otimes \vol + \vol \otimes \Phi_\beta + (\Phi\otimes\Phi)_{\pi_{C*} (dg \wedge \beta) }.
$$
where $\vol^A(U):= \vol(A\cap U)$.
\end{proof}

\subsection{The transfer principle} \label{transfer section}
The following is an instance of the general transfer principle of Howard \cite{howard93}. 

Let $M_\pm$ be two Riemannian manifolds of dimension $n$, and $G_\pm $ be Lie groups acting isotropically on $M_\pm$. Assume further that the subgroups $H_\pm$ fixing chosen points $o_\pm \in M_\pm$ are isomorphic, and that there is an isometry
\begin{equation}\label{transfer hypothesis}
\iota:T_{o_+} M_+ \to T_{o_-} M_- 
\end{equation}
intertwining the actions of $H_\pm$.
We identify $H_\pm$ with a common model $H$. Let $K\subset H$ be the subgroup of points fixing a chosen point $\bar o_\pm\in SM_\pm$ with $\pi_\pm(\bar o_\pm) = o_\pm$, where $\pi_\pm: SM_\pm \to M_\pm$ is the projection. Since the actions of $G_\pm$ are isotropic, we may assume that $\iota (\bar o_+) = \bar o_-$.
To simplify notation we denote the points $o_\pm,\bar o_\pm$ by $o,\bar o$. 

The tangent spaces to the sphere bundles may be decomposed
$$
T_{\bar o} SM_\pm = P_\pm\oplus V_\pm
 $$ 
into the horizontal (with respect to the Riemannian connection) and vertical subspaces. Thus there are canonical isomorphisms
$$
P_\pm \simeq  T_oM_\pm ,\quad V_\pm \simeq \bar o ^\perp \subset T_o M_\pm
$$
and therefore $\iota$ induces a $K$-equivariant isomorpism $\bar \iota: T_{o} SM_+ \to T_{o} SM_- $.

This gives an isomorphism of exterior algebras
$$
\bar \iota^*: \Omega^*(SM_-)^{G_-} \simeq \Lambda^*(T_{\bar o}SM_-)^K \to \Lambda^*(T_{\bar o}SM_+)^K\simeq \Omega^*(SM_+)^{G_+}.
$$
Since the contact form $\alpha$ restricts to the inner product with $\bar o$ on the $P$ factor, and to $0$ on $V$, it follows that $\bar \iota^* \alpha = \alpha$. Although $\bar \iota^*$ is {\it not} an isomorphism of {\it differential} algebras (i.e. does not intertwine $d$), nevertheless $\bar \iota^* d\alpha  =d\alpha$: this can be seen directly since in each case $d\alpha = \sum_{i=1}^{n-1} \theta_i \wedge \tilde\theta_i$, where $\theta_i,\tilde \theta_i$ are orthonormal coframes for $\bar o^\perp \subset P$ and for $V$ that correspond under the natural isomorphism--- in other words, $d\alpha$ corresponds to the natural symplectic form on the cotangent bundles. Therefore $\bar\iota^*$ induces a natural isomorphism
\begin{equation}
\tildiota: \curv^{G_-}( M_-) \to \curv^{G_+}(M_+)
\end{equation}
via the identifications
\begin{equation*}
\curv^{G_\pm}( M_\pm)\simeq \Omega^{n-1}(SM_\pm)^{G_\pm}/(\alpha, d\alpha).
\end{equation*}

\begin{theorem} \label{transfer} If there exists an isometry \eqref{transfer hypothesis} as above then the following diagram commutes:
\begin{equation}
\begin{CD} \curv^{G_-}(M_-)& @>{\tildiota}>> &  \curv^{G_+}(M_+)\\
@V{\tildk_{G_-}}VV & &  @V{\tildk_{G_+}}VV \\
\curv^{G_-}(M_-)\otimes \curv^{G_-}(M_-)& @>{\tildiota\otimes \tildiota}>>  & \curv^{G_+}(M_+)\otimes\curv^{G_+}(M_+)\\
\end{CD}
\end{equation}
\end{theorem}
\begin{proof} Note that each $G_\pm$ may be identified with a subbundle $\mathcal G_\pm$ of the orthonormal frame bundle $\mathcal F_\pm$ of $M_\pm$: select an arbitrary orthonormal frame $f$ at some point $o\in M_\pm$, and take $\mathcal G_\pm$ to be the $G_\pm$ orbit of $f$. Since $G_\pm$ acts effectively, the induced map is a diffeomorphism.

Recall \cite{bish-crit} that the Riemannian metric on $M_\pm$ induces a canonical invariant horizontal distribution $\mathcal D$ on $\mathcal F_\pm$, with each plane $\mathcal D_f\subset T_f \mathcal F_\pm$  linearly isomorphic to $T_{\pi f} M_\pm$, where $\pi: \mathcal F_\pm \to M_\pm$ is the projection. A modification of the Sasaki metric endows $\mathcal F_\pm$ with a Riemannian structure $g$, where a) each $\mathcal D_f \perp \pi^{-1}\pi f$, b) the restriction of $g$ to each fiber $\pi^{-1}x$ is induced by the standard invariant metric on $SO(n)\simeq \pi^{-1}x \subset \mathcal F_\pm$, and c) the restriction of $\pi_*$ to $\mathcal D_f$ is an isometry to $T_{\pi f} M$.

For $f \in \mathcal G_\pm$ consider the orthogonal projection, with respect to this Sasaki metric, of $\mathcal D_f $ to $T_f \mathcal G_\pm \subset T_f \mathcal F_\pm$. This yields a distribution $\mathcal M_\pm$ on $\mathcal G_\pm$, complementary to the tangent spaces of the fibers of $\mathcal G_\pm$, and clearly invariant under the action of $G_\pm$ (since this action is by isometries). Putting $\frak m_\pm:= \mathcal M_{\pm, o} \subset T_o \mathcal G_\pm \simeq T_e G_\pm = \frak g_\pm$, it follows
 that at the level of Lie algebras there is a natural decomposition
\begin{equation}
\frak g_\pm = \frak h \oplus \frak m_\pm
\end{equation}
where $\frak m_\pm$ is naturally identified with $T_oM_\pm$ under the projection map, and is invariant under the adjoint action of $H$ (i.e. the homogeneous space $G/H$ is {\it reductive}). Furthermore the isometry $\iota$ induces an $H$-equivariant isomorphism $\frak m_+\to \frak m_-$.

We claim that the maps $\tilde k_{G_\pm}$ depend only on the data above. To see this, for simplicity of notation we drop the $\pm$, and consider the diagram of derivatives of  \eqref{main diagram} over the point $(\bar o ,\bar o) \in SM\times SM$ for the pairs $(M_\pm,G_\pm)$:
\begin{equation}\label{derived diagram}
\begin{CD}
\left.TE\right|_{\bar F} &@>>> &\left.T\left(G\times SM \right)\right|_{F} \\
@VVV & & @VVV\\
T_{\bar o}SM \times T_{\bar o}SM&@>>>& T_oM \times T_oM
\end{CD}
\end{equation}
where $ F= H\times H/K,\bar F$ are the fibers over $( o,o)$, $(\bar o,\bar o)$. The vertical bundles have fiber $TH \times T(S_oM)= TH \times T(H/K)$, and the diagram is again a cartesian square. Writing this in terms of the Lie algebras,
\begin{equation}\label{Lie diagram}
\begin{CD}
\left.TE\right|_{\bar F} &@>>> &TH \oplus \frak m \oplus T(H/K) \oplus \frak m \\
@VVV & & @VVV\\
 \frak h/\frak k \oplus \frak m\oplus \frak h/\frak k \oplus \frak m &@>>>& \frak m\oplus \frak m
\end{CD}
\end{equation}
since $\left.TG\right|_H= TH \oplus \frak m$. The maps on the bottom and on the right are the obvious projections.

These diagrams are $(H\times H)$-equivariant, and the $H\times H$ actions depend only on the structure of $\frak m$ as an $H$-module. In particular we may re-insert the $\pm$, and in the diagram obtained from the obvious maps between the two diagrams is commutative and $(H\times H)$-equivariant. The restrictions of the $G$-invariant forms on $SM$ are precisely the $H$-invariant sections of the exterior algebra bundles of the various bundles occurring here. Our convention on volume forms dictates that the distinguished volume form of $G$ restricts to the product of the probability volume form on $H$ with the natural volume form on $\frak m$ arising from the identification with $T_oM$. In other words, the natural maps between the $\pm$ diagrams respect the mapping $\bar \iota$.

Finally, it is clear that the spaces (or currents) $C$ from \eqref{def C} also correspond under these maps, hence the maps between the $\pm$ spaces also intertwine the fiber integrals over $C$, which is the assertion of the theorem.
\end{proof}

Thm. \ref{transfer} implies, for example, that in some sense the kinematic formulas of all three space forms $\Rn, S^n, H^n$ are the same, in the sense that there is a canonical identification of the spaces of invariant curvature measures on $S^n$ (or $H^n$) and $\Rn$ which moreover intertwines the respective kinematic operators $\tilde k$: these spaces are the quotients $G_\lambda/SO(n)$ for $\lambda= 0, 1,-1$ respectively, where $G_0 = SO(n)\times \Rn, G_1 = SO(n+1), G_{-1} = SO(n,1)$. We take $H= SO(n) \supset K:= SO(n-1)$. Their Lie algebras are the subalgebras
\begin{equation}\label{lie algebras}
\frak g_\lambda =\left\{\left[
\begin{matrix}
0 & a_{01} & \dots & a_{0n}\\
a_{10} & &  & \\
\dots & & h& \\
a_{n0} & & & 
\end{matrix}
\right]: h \in \frak{so}(n), \ a_{0i} + \lambda a_{i0} =0
\right\}
\end{equation}
of $\frak{gl}(n+1)$.
The subspace $\frak m = \{h=0\}$.

The complex space forms $\Cn$ (with the restricted isometry group $\barun$), $\CP^n, \C H^n$ admit a similar description, so again it follows that the integral geometry of these spaces at the level of curvature measures is independent of the ambient curvature. However, the case of the real space forms is uniquely simple, due to the fact that the map from $\barson$-invariant curvature measures to valuations is an isomorphism.

\section{Integral geometry of euclidean spaces via Alesker theory}
\subsection{Survey of valuations on finite dimensional real vector spaces}\label{survey}

The recent work of S. Alesker revolves around a deepened understanding of convex valuations.
 Given a finite dimensional real vector space $W$ of dimension $n$, consider the space $\Val = \Val(W)$ of continuous translation-invariant convex valuations on $W$ . For convenience we will assume that $W$ is endowed with a euclidean structure, although this device may be removed by inserting the dual space $W^*$ appropriately. The valuation $\varphi \in \Val(W)$ is said to have {\bf degree} $k$ and {\bf parity }$\eps = \pm1$ if
\begin{align*}
 \varphi(tK)& = t^k \varphi (K) , \quad t >0,\\
 \varphi (-K) &= \eps \varphi (K)
 \end{align*}
 for all $K\in \K(W)$. Denote the subspace of valuations of degree $k$ and parity $\eps$ on $W$ by $\Val_{k,\eps}(W) \subset \Val(W)$.

Putting
$$
\norm{\varphi} := \sup\{\varphi(K): K\in \K(W), K \subset B_1\},
$$
where $B_1$ is the closed unit ball in $W$, gives $\Val(W)$ the structure of a Banach space. The group $GL(W)$ acts on $\Val(W)$ by $g \cdot \varphi (K):= \varphi (g^{-1}K)$, and this action stabilizes each $\Val_{k,\eps}$. Put $\valsm(W)$ for the subspace of {\bf smooth valuations}, i.e. the space of valuations $\varphi$ such that the map $GL(W) \to \Val, g \mapsto g \varphi$, is smooth. General theory (cf. \cite{ale01}) ensures that $\valsm$ is dense in $\Val$.

The starting point for Alesker's approach is

\begin{theorem}[Irreducibility Theorem \cite{mcmullen77,ale01}]\label{irred thm} As a $GL(W)$-module, the decomposition of $\Val(W)$ into irreducible components is
\begin{equation}\label{mcmalesk decomp}
\Val(W) = \bigoplus_{k=0,\dots,n; \ \eps = \pm 1} \Val_{k,\eps} (W).
\end{equation}
Furthermore $\Val_0, \Val_n$ are both 1-dimensional, spanned by the Euler characteristic $\chi$ and the volume $\vol$ respectively.
\end{theorem}

Irreducibility here means that each $\Val_{k,\eps}(W)$ admits no nontrivial, closed, $GL(W)$-invariant subspace.

For $A\in \K(W)$ we define the valuation $\mu_A\in \Val(W)$ by
\begin{equation}\label{def mu}
\mu_A(K):= \vol(A+K) = \int_W \chi((x-A) \cap K) \, dx.
\end{equation}
The integrand of course takes only the values 0 and 1, depending on whether or not the intersection is empty.

Given an even valuation $\varphi$ of degree $k$ we say that a signed measure $m_\varphi$ on $\Gr_k(W)$ is a {\bf Crofton measure} for $\varphi$ if
$$
\varphi (A) = \int_{\Gr_k} \vol_k(\pi_E(A)) \, dm_\varphi(E).
$$
If $k = 1 $ or $n-1$ then $m_\varphi$ is uniquely determined by $\varphi$, but not in the remaining cases $2\le k\le n-2$. Prop. \ref{general crofton} means that the Haar measure on $\Gr_k$ is a Crofton measure for $\mu_k, k =0,\dots,n$.

Klain \cite{kl00} proved that an even valuation $\varphi \in \Val_{k,+}$ is uniquely determined by its {\bf Klain function}
$$
\kl_\varphi:\Gr_k \to \R,
$$
defined by the condition that the restriction of $\varphi$ to $E\in \Gr_k$ is equal to $\kl_\varphi(E) \left.\vol_k\right|_E$.

Our next statement summarizes the most important implications of Theorem \ref{irred thm} for integral geometry.

\begin{theorem}[Alesker \cite{ale03b, ale04, ale04a, aleicm, ale07}]\label{alesker package}
\begin{enumerate}
\item \label{span} 
\begin{equation}\label{eq:val as sum}
\langle \vol_n\rangle \oplus \langle \Psi_\gamma: \gamma \in \Omega^{n-1}(SW)^W\rangle = \bigoplus_{k=0}^n\valsm_k(W).
\end{equation}
\item \label{crofton measures} Every smooth even valuation of degree $k$ admits a smooth Crofton measure.
\item\label{mcm conj} $ \langle \mu_A: A \in \K(W)\rangle$ is dense in $\Val(W)$.
\item There is a natural continuous commutative graded product on $\valsm(W)$ such that 
\begin{equation}\label{product}
(\mu_A \cdot \varphi)(K)= \int_W \varphi((x-A) \cap K)\, dx.
\end{equation}
In particular, the multiplicative identity is the Euler characteristic $\chi$.
\item \label{af transform} (Alesker Fourier transform) There is a natural linear isomorphism  $\ \widehat{}: \valsm_{k,\eps} \to \valsm_{n-k,\eps}$ such that $\ \widehat {{ \widehat { \varphi}}}= \eps\varphi$. For smooth even valuations, this map is given in terms of Crofton measures by
\begin{equation}
m_{\hat \varphi} = \perp_* m_{\varphi}
\end{equation}
where $\perp:\Gr_k \to \Gr_{n-k}$ is the orthogonal complement map. Equivalently,
\begin{equation}
\kl_{\hat\varphi}= \kl_\varphi \circ \perp.
\end{equation}
 \item \label{poincare duality} The product satisfies Poincar\'e duality, in the sense that the pairing $$(\varphi,\psi):= \text{ degree $n$ part of }\varphi \cdot \psi \in \Val_n \simeq \R$$ is perfect.
The Poincar\'e pairing is invariant under the Fourier transform:
\begin{equation}\label{pd and fourier}
(\hat \varphi,\hat \psi) = (\varphi,\psi).
\end{equation}
\item\label{hard lefschetz} (Hard Lefschetz) The degree $1$ map $L'\varphi:= \mu_1\cdot \varphi$ satisfies the hard Lefschetz property: for $k\le \frac n 2$ the map 
$$
(L')^{n-2k}: \valsm_{k} \to \valsm_{n-k}
$$
is a linear isomorphism.
\item \label{G}  If $G \subset SO(W)$ acts transitively on the sphere of $W$ then the subspace of $G$-invariant valuations has $\dim \valg (W)<\infty$. Furthermore $\valg(W) \subset \valsm(W)$.
\end{enumerate}
\end{theorem}
%
%

{\bf Remarks:} 1) Theorem \ref{irred thm} implies that the space on the left hand side of \eqref{eq:val as sum} is dense in $\Val(W)$; that this space consists precisely of the smooth valuations follows from general representation theory (the Casselman-Wallach theorem).

2) The groups $G$ as in \eqref{G} above have been classified (cf. Alesker's lecture notes). In fact it is true that, module the volume valuation, every $\varphi \in \valg$ is given by integration against the normal cycle of some $G$-invariant form on the sphere bundle. This implies immediately that $\dim \valg <\infty$.

Intertwining the Fourier transform and the product we obtain the {\bf convolution} on $\valsm$:
\begin{equation}\label{def conv}
\varphi * \psi:= \widehat{\hat \varphi \cdot \hat \psi}.
\end{equation}
Recall that if $A_1,\dots,A_{n-k}\in \K$ then
$$
V(A_1,\dots,A_{n-k}, B[k]):= \frac{k!}{n!}\left.\frac{\partial^{n-k}}
{\partial t_1\dots \partial t_{n-k}}\right|_{t_1=\dots=t_{n-k} =0} \mu_{\sum t_iA_i}(B)
$$
is the associated {\bf mixed volume}, which is a translation-invariant valuation of degree $k$ in $B$. Here $B[k]$ denotes the $k$-tuple $(B,\dots,B)$.

\label{convolution = addition}\begin{theorem}[\cite{befu06, ale07}]
\begin{align}
\label{conv 1}\mu_A* \varphi &= \varphi (\cdot + A)\\
\label{conv 2} V(A_1,\dots,A_{n-k}, \cdot) * V(B_1,\dots,B_{n-l}, \cdot)&= \frac{k!l!}{n!}V(A_1,\dots,A_{n-k}, B_1,\dots,B_{n-l},\cdot).
\end{align}
\end{theorem}

\begin{exercise} \eqref{conv 2} follows from \eqref{conv 1}.
\end{exercise}

\begin{corollary} \label{def lambda}
Define the degree $-1$ operator $\Lambda'$ by
$$
\Lambda' \varphi:=\widehat{(L'\hat \varphi)} .
$$
Then
\begin{equation}\label{def lambda'}
\Lambda' \varphi  = \mu_{n-1}*\varphi=\frac 1 2\left.\frac d{dr}\right|_{r=0} \varphi (\cdot + B_r).
\end{equation}
This may in turn be expressed in terms of the underlying differential forms as
$$
 \Lambda' \Psi_\theta = \frac 1 2\Psi_{\mathcal L_T\theta}.
$$
where $T(x,v) := (v,0)$ is the Reeb vector field of $T\R^n$ and $\mathcal L$ is the Lie derivative.
\end{corollary}

It is convenient to renormalize the operators $L',\Lambda'$ by putting
\begin{align}
L&:= \frac{2\omega_{k}}{\omega_{k+1}} L'\\
\Lambda&:= \frac{2\omega_{n-k}}{\omega_{n-k+1} } \Lambda'
\end{align}
on valuations of degree $k$. Thus
\begin{equation}\label{eq:L lambda}
\Lambda \varphi=\widehat{(L\hat \varphi)} .
\end{equation}
In view of relation \eqref {powers of t} below, these act on the intrinsic volumes as
\begin{align}
L\mu_k&= (k+1)\mu_{k+1}\\
\Lambda\mu_k&= (n-k+1) \mu_{k-1}.
\end{align}

\subsection{Constant coefficient valuations}
The normal cycle of a smooth submanifold is its unit normal bundle, so it makes sense to think of $N(A)$ for general sets $A$ as a generalization of this concept. It is also convenient to consider the analogue $N_1(A)\subset T\R^n\simeq \Rn\times \Rn$ of the bundle of unit normal balls, obtained by summing $A \times\{0\}$ and the image of $N (A) \times [0,1]$ by the homothety map in the second factor. Thus 
$$
\partial N_1(A) = N(A).
$$
 If a differential form $\varphi \in \Omega^{n-1}(S\Rn)$ extends to a smooth form on all of $\Rn\oplus \Rn$ then Stokes' theorem yields
$$
\Psi_\varphi  := \int_{N(\cdot)} \varphi =\int_{N_1(\cdot)} d\varphi .
$$
If $d\varphi$ happens to have constant coefficients (i.e. is invariant under translations of both the base and the fiber) then a number of simplifications ensue.
\begin{definition}\label{ccvs}  $\mu\in \valsm(V)$ is a {\bf constant coefficient valuation} if there exists $\theta \in \Lambda^n(V\oplus V)$ such that
$$
\mu= \overpsi_\theta:=\int_{N_1(\cdot)} \theta.
$$
We denote by $CCV\subset \Val(V)$ the finite-dimensional vector subspace consisting of all constant coefficient valuations.
\end{definition}

\begin{exercise} Every constant coefficient valuation is smooth and even.
\end{exercise}

This concept only makes sense with respect to the given euclidean structure on $V$--- the spaces of constant coefficient valuations associated to different euclidean structures are different. Using Stokes' theorem it is easy to see that $\valson \subset CCV$, i.e. that the intrinsic volumes $\mu_i$ are all constant coefficient valuations. We will see below that if  $V=\Cn$ then this is also true of all valuations invariant under $U(n)$.

When restricted to constant coefficient valuations, the actions of $L, \Lambda$ and the Fourier transform admit the following simple algebraic model. Consider adapted coordinates $$x_1,\dots,x_n,y_1,\dots,y_n$$ for $\Rn\oplus \Rn$, so that $\omega=\sum dx_i \wedge dy_i$ is the usual symplectic form. 
Put $j:\Lambda^*(\Rn\oplus \Rn)\to \Lambda^*(\Rn\oplus \Rn)$ for the algebra isomorphism that interchanges the coordinates:
$$
j(dx_i) = dy_i, \quad j(dy_i) = dx_i, \quad i =1,\dots,n.
$$
There are derivations
$\ell, \lambda : \Lambda^*(\Rn\oplus \Rn)\to \Lambda^*(\Rn\oplus \Rn)$, of degrees $\pm 1$ respectively, determined by
\begin{align}
\notag \ell(dy_i) := dx_i, &\quad \ell(dx_i) := 0\\
\notag \lambda(dx_i) := dy_i, &\quad \lambda(dy_i) := 0,\\
\label{intertwine} j\circ \ell &= \lambda \circ j.
\end{align}
In fact, putting $m_t(x,y):= (x+ty,y)$
\begin{equation}\label{derivative formula}
\lambda\varphi = \left.\frac {d}{dt}\right|_{t=0} m_t^*\varphi
\end{equation}
and a similar formula holds for $\ell$. It is easy to see that for a monomial $\varphi \in \Lambda^*(\Rn\oplus\Rn)$
\begin{equation}\label{bracket1}
[\ell,\lambda] \,\varphi = (r-s) \varphi
\end{equation}
where $r,s$ respectively denote the number of $dx_i$ and $dy_i$ factors in $\varphi$. Putting $H$ for the operator $\varphi \mapsto (r-s)\, \varphi$, clearly
\begin{equation}\label{bracket2}
[H,\ell] = 2\ell, \quad [H,\lambda] =- 2\lambda.
\end{equation}
In other words,
\begin{lemma} \label{model sl2 representation}
Let $X,Y,H$ with $[X,Y]=H, [H,X]=2X, [H,Y]=-2Y$ be generators of
  $\mathfrak{sl}(2,\mathbb{R})$. The map 
\begin{align*}
H & \mapsto r-s\\
X & \mapsto \ell \\
Y & \mapsto \lambda 
\end{align*}
defines a representation of $\mathfrak{sl}(2,\mathbb{R})$ on $\Lambda^*(\Rn\oplus \Rn)$. 
\end{lemma}

The subspace $\Lambda^n(\Rn \oplus \Rn)$ is naturally graded by the number of factors $dx_i$ that appear.

\begin{proposition}\label{ccv model}
\begin{enumerate}
\item \label{surjective}The surjective map $\overpsi: \Lambda^n(\Rn\oplus \Rn) \to CCV$ is graded. The kernel of $\overpsi$ consists precisely of the subspace of multiples of the symplectic form.
\item \label{intertwine2}  The operators $j,\lambda,\ell$ induce (up to scale) the operators~~$\widehat{}\ ,\Lambda, L$ on $CCV$, via the formulas in degree $k$
\begin{align*}
\label{hat psi}
\widehat{}\,\circ{\overpsi} &= \frac{\omega_{n-k}}{\omega_k} \, \overpsi\circ j \\
\Lambda\circ \overpsi&= \frac{\omega_{n-k}}{\omega_{n-k+1}}\overpsi \circ{\lambda}\\
L\circ \overpsi&= \frac{\omega_{n-k}}{\omega_{n-k-1}}\overpsi\circ{\ell}\\
\end{align*}
and satsifying the relation
$$
\Lambda \circ\, \widehat{} \ = \ \widehat{}\circ L.
$$
\item \label{thm_sl2_representation}The map 
\begin{align*}
H & \mapsto 2k-2n\\
X & \mapsto L\\
Y & \mapsto \Lambda 
\end{align*}
defines a representation of $\mathfrak{sl}(2,\mathbb{R})$ on $CCV$. 
\end{enumerate}
\end{proposition}
\begin{proof} \eqref{surjective}:
The first assertion is obvious, and the second follows at once from Lemma \ref{curv meas ker}.

\eqref{intertwine2}:
Noting that $j$ takes the symplectic form to $-1$ times itself, and that $\ell, \lambda$ annihilate it, the first assertion follows at once. The first formula follows at once from the definition of the Alesker Fourier transform. To prove the second formula, let $\varphi \in \Lambda^n(\Rn\oplus \Rn)$ and let $\theta\in \Omega^{n-1}(\Rn \times \Rn)$  be a primitive. Then by \eqref{def lambda'} and \eqref{derivative formula}
\begin{align*}
\Lambda'\overpsi_{\varphi}&=\Lambda' \Psi_\theta\\
&=\Lambda' \int_{N(\cdot)} \theta \\
&= \frac 1 2 \int_{N(\cdot)} \mathcal L_T\theta\\
&= \frac 1 2 \int_{N_1(\cdot)} d \mathcal L_T\theta\\
&= \frac 1 2 \int_{N_1(\cdot)} \mathcal L_Td\theta\\
&=\frac 1 2 \overpsi_{\lambda \varphi}
\end{align*}
from which the second formula follows at once. The fourth formula is \eqref{eq:L lambda}, and the third follows from that relation and  \eqref{intertwine}.

\eqref{thm_sl2_representation}: This follows at once by calculation from conclusion \eqref{intertwine2} and Lemma \ref{model sl2 representation}.
\end{proof}

\begin{problem} What is the maximal subspace of $\valsm$ for which assertion \eqref{thm_sl2_representation} holds? (Alesker has shown that it does not hold for the full algebra $\valsm$.)
\end{problem}

\begin{exercise} If $\mu$ is a constant coefficient valuation  of degree $k$ and $P\subset \Rn$ is a convex polytope then
\begin{equation}\label{face decomp}
\mu(P) = \sum_{F \in P_k} \kl_\mu(\la F\ra) \vol_k(F)\angle (P,F) 
\end{equation}
where $P_k$ is the $k$-skeleton of $P$ and $\angle (P,F) $ is the appropriately normalized exterior angle of $P$ along $F$.
\end{exercise}

\begin{problem} Characterize the $\mu \in \Val_k^+$ satisfying \eqref{face decomp} for every convex polytope $P$. (Note that every even valuation of degree $n-1$ satisfies \eqref{face decomp}, but the space of such valuations is infinite-dimensional.)
\end{problem}
\begin{problem} As we will see below, the algebras $\valson(\Rn)$ of $\barson$-invariant valuations on $\Rn$ and $\valun(\Cn)$ of $\barun$-invariant valuations on $\Cn$ are subspaces of $CCV$. Classify the subalgebras of the constant coefficient valuations. Do they constitute an algebra? If not, what algebra do they generate?
\end{problem}

\subsection{The {\it ftaig} for isotropic structures on euclidean spaces}

We note two important consequences of Thm. \ref{alesker package}. First, assertion \eqref{span} (together with Corollary \ref{cor:alt ker} below) implies that every $\varphi \in \valsm(W)$ can be applied not only to elements of $\K(W)$ but also to any compact set that admits a normal cycle, e.g. smooth submanifolds and submanifolds with corners. Second, recalling how Hadwiger's Theorem \ref{hadwiger thm} implies the existence of the classical kinematic formulas, Thm. \ref{alesker package} \eqref{G} implies the existence of  kinematic formulas for euclidean spaces $V$ under an isotropic group action. For $G\subset O(V)$ put $\barg:= G\ltimes V$ for the semidirect product of $G$ with the translation group. Thus $(V,\barg)$ is isotropic iff $G$ acts transitively on the sphere of $V$.

\begin{proposition}[\cite{fu90, ale03b}] \label{vs kinematic}
In this case there are linear maps
$$
k_G, a_G: \valg \to \valg \otimes \valg
$$
such that for $K,L \in \K(V)$ and $\varphi \in \valg$
\begin{align*}
 k_G(\varphi)(K,L)&=\int_{\barg} \varphi(K\cap \bar g L) \, d\bar g \\
  a_G(\varphi)(K,L)&=\int_{G} \varphi(K+ g L) \, d g
\end{align*}
The {\bf kinematic operator} $k_G$ is related to $\tilde k_G$ via the commutative diagram
\begin{equation}
\begin{CD} \curv^{G}(V)&  @>{\tildk_{G}}>> & \curv^{G}(V)\otimes\curv^{G}(V) \\
@V{\Psi}VV &  & @V{\Psi}\otimes{\Psi}VV \\
\valg(V) & @>{k_{G}}>>  &\valg(V)\otimes \valg(V)
\end{CD}
\end{equation}

\end{proposition} 
\begin{proof} The statements about $k_G$ are a direct consequence of Thm. \ref{kf for cm}. Invoking \eqref{G} of Thm. \ref{alesker package}, the proof of the existence of $a_G$ (and another proof of the existence of $k_G$) follows precisely the proof of Thms. \ref{bkfs} and \ref{akfs} above.
\end{proof}

The next theorem may rightly be called the {\bf Fundamental Theorem of Algebraic Integral Geometry} for euclidean spaces. Before stating it we need to add some precision to the statement \eqref{poincare duality} of Thm. \ref{alesker package} by specifying the isomorphism $\Val_n(\Rn) \simeq \R$ to take Lebesgue $\vol_n \simeq 1$.

\begin{theorem}\label{ftaig} Let $p:\valg \to {\valg}^*$ denote the Poincar\'e duality map from \eqref{poincare duality} of Thm. \ref{alesker package}, $m_G: \valg \otimes \valg\to \valg$ the restricted multiplication map, and $m_G^*: {\valg}^*\to{\valg}^* \otimes {\valg}^*$ its adjoint. Then the following diagram commutes:
\begin{equation}
\begin{CD} 
\valg(V) & @>{k_G}>>  &\valg(V)\otimes \valg(V)\\
@V{p}VV &  & @V{p}\otimes{p}VV \\
{\valg}^*(V)&  @>{m_G^*}>> & {\valg}^*(V)\otimes{\valg}^*(V) 
\end{CD}
\end{equation}

The same is true if $k_G,m_G$ are replaced by $a_G, c_G$ respectively, where $c_G$ is the convolution product. In particular
\begin{equation}\label{kg ag}
a_G = (\,\hat{}\otimes \hat{}\, )\circ k_G\circ \hat{}.
\end{equation}
\end{theorem}
\begin{proof} The dual space ${\valg}^*$ is a $\valg$-module by
$$
\la(\alpha\cdot \beta^*),\gamma\ra := \la \beta^*, \alpha \cdot \gamma\ra.
$$
With this definition it is clear that $p$ is map of $\valg$ modules, and it is easy to check that $m_G^*$ is multiplicative in the sense that
$$
m_G^*(\alpha\cdot \beta^*) = (\alpha\otimes \chi) \cdot m_G^*(\beta) = (\chi\otimes\alpha) \cdot m_G^*(\beta).
$$

On the other hand, by \eqref{mcm conj} and \eqref{G} of Thm. \ref{alesker package}, the valuations
\begin{equation}\label{isom}
\mu_A^G:= \int_{\barg} \chi(\, \cdot\, \cap \bar g A) \, d\bar g = k_G(\chi)(\cdot, A)
\end{equation}
span $\valg$. Thus we may check the multiplicativity of $k_G$ by computing
\begin{align*}
k_G(\mu_A^G\cdot \varphi)(B,C) &= \int_{\barg} (\mu_A^G\cdot \varphi) (B\cap \bar g C) \, d \bar g\\
&= \int_{\barg}\int_{\barg}  \varphi (B\cap \bar h A\cap \bar g C)\, d\bar h \, d \bar g\\
&= \int_{\barg} k_G(\varphi) (B\cap \bar h A, C)\, d\bar h \\
&=\left[ (\mu_A^G\otimes \chi)\cdot k_G(\varphi)\right](B,C)
\end{align*}
by Fubini's theorem, for $\varphi \in \valg$.

Thus it remains to show that $(p\otimes p)(k_G(\chi)) = m_G^*(p(\chi))$. We may regard these elements of $\Val^{G*} \otimes \Val^{G*}$ as lying in $\Hom(\valg,\Val^{G*})$ instead. Notice that in these terms both elements are graded maps in a natural way. Furthermore \eqref{isom} and the multiplicativity of $k_G$ implies that $(p\otimes p)(k_G(\chi))$ is an invertible map of $\valg$ modules, and it is clear that the same is true of  $m_G^*(p(\chi))$. Since $\dim(\Val_0)=1$ it follows that the two must be equal up to scale.

To determine the scaling factor we compare $(p\otimes p)(k_G(\vol)) ,\ m_G^*(p(\vol))$. But using the facts
$$
p(\vol) =\chi^*,\quad k_G(\vol)=\vol\otimes \vol
$$
(the latter follows from the conventional normalization \eqref{eq:measure convention} of the Haar measure $dg$) it follows that the scaling factor must be 1. Here $\chi^*$ is the dual element evaluating to $1$ on $\chi$ and annihilating all valuations of positive degree.
\end{proof}

In more practical terms Thm. \ref{ftaig} may be summarized in the following statements:
\begin{enumerate}
\item  Let $\nu_1,\dots,\nu_N$ and $\phi_1,\dots,\phi_N$ be two bases for $\valg$, and consider the $N\times N$ matrix
$$
M_{ij}:= \la p(\nu_i),\phi_j\ra.
$$
Then
\begin{equation}
k_G(\chi) = \sum_{ij}(M^{-1})_{ij} \nu_i\otimes \phi_j.
\end{equation}
In other words $k_G(\chi)=p^{-1}\in \Hom(\Val^{G*},\Val^G)$. Hence, $k_G$ also determines the restricted product $m_G$. 
\item $k_G$ is multiplicative, in the sense that $k_G(\varphi \cdot \mu) = (\varphi \otimes \chi)\cdot k_G(\mu)$.
\end{enumerate}

Bernig has observed  
\begin{proposition}\label{valg subset val+} If $G$ acts isotropically on $\Rn$ then
$\valg(\Rn) \subset \Val_+(\Rn)$.
\end{proposition}
No general explanation for this fact is known. With \eqref{pd and fourier} this implies that if $\phi_i = \hat \nu_i$ above then the matrix $M$ is symmetric, and hence so is $M^{-1}$.

\subsection{The classical integral geometry of $\Rn$}\label{ig of rn}
 Modulo $\alpha, d\alpha$, the space of invariant forms $\Omega^{n-1}(S\Rn)^{SO(n)}=\la \kappa_0,\dots,\kappa_{n-1}\ra$, corresponding to the elementary symmetric functions of the principal curvatures of a hypersurface. With Thm. \ref{alesker package} this implies Hadwiger's theorem \ref{hadwiger thm}. Furthermore

\label{structure of valson}\begin{theorem} As an algebra, $\valson(\R^n)\simeq \R[t]/(t^{n+1})$, with 
\begin{equation}\label{powers of t}
t^i = \frac{i! \omega_i}{\pi^i}\mu_i = \frac{2^{i+1}}{\alpha_i}\mu_i.
\end{equation}
\end{theorem}
\begin{proof} Put $t:= \int_{\barGr_{n-1}} \chi(\cdot \cap P)\, dP$. Then by the definition of the Alesker product
$$
t^2 = \int_{\barGr_{n-1}} t(\cdot \cap P)\, dP= \int\int_{(\barGr_{n-1})^2} \chi(\cdot \cap P\cap Q)\,dQ dP
=\int_{\barGr_{n-2}} \chi(\cdot \cap R)\, dR
$$
etc. By \eqref{affine general crofton}, it follows that the powers of $t$ are multiples of the corresponding $\mu_i$, which establishes the first assertion. 

To determine the coefficients relating the $t^i$ and the $\mu_i$ we apply the transfer principle Thm. \ref{transfer} to the isotropic pairs $(\Rn,\barson)$ and $(S^n,SO(n+1))$. Let $\Psi_i= \frac 2{\alpha_i} \mu_i$ and $\Psi_i'$ its image under the transfer (i.e. if $\psi_i'$ is the image of the corresponding curvature measure under the transfer map then $\Psi_i'(A) = \psi_i'^A(A)$ for $A\subset S^n$). Then
$$
\Psi_i'(S^j) = 2 \delta_i^j
$$
and the kinematic formula for $(S^n,SO(n+1))$ is
$$
k_{S^n}(\Psi_c') (S^a,S^b)= \int_{SO(n+1)} \Psi_c'(S^a\cap g S^b)\, dg = 2\alpha_n \delta_c^{a+b-n}
$$
under our usual convention for the Haar measure on the group. Thus the template method implies that
$$
k_{S^n}(\Psi'_c) = \frac{\alpha_n} 2 \sum_{a+b = n+c} \Psi'_a\otimes\Psi'_b.
$$
Now the multiplicativity of $k_{\Rn}$, together with Thm. \ref{transfer}, yields
$$
\frac{\alpha_n} 2(\Psi_c\otimes \chi)\cdot  \sum_{a+b = n} \Psi_a\otimes\Psi_b= 
(\Psi_c\otimes \chi)\cdot k_{\Rn}(\chi) = k_{\Rn}(\Psi_c) = \frac{\alpha_n} 2 \sum_{a+b = n+c} \Psi_a\otimes\Psi_b.
$$
It follows that $\Psi_a\cdot \Psi_b \equiv \Psi_{a+b}$. It is convenient however to take 
\begin{equation}
t:= 2\Psi_1
\end{equation}
whence the second relation of \eqref{powers of t} follows. The first then follows from the identity
$$
\omega_n\omega_{n+1} = \frac{2^{n+1}\pi^n}{(n+1)!}.
$$
\end{proof}

Adjusting the normalization of the Haar measure $dg$ and unwinding Thm. \ref{ftaig}, it follows that 
\begin{equation}
\kson(t^c) = \sum_{a+b = n+c} t^a \otimes t^b, \quad c = 0,\dots, n.
\end{equation}
In fact  $dg$ is chosen so that $dg(\{g: go \in S\}) =\frac{2^{n+1}}{\alpha_n}\vol_n(S) $, as may be seen by examining the leading term  $t^c \otimes t^n = \frac{2^{n+1}}{\alpha_n} t^c \otimes \vol_n$ on the right.
This yields the first assertion of Theorem \ref{nijenhuis}.

\begin{corollary}
$$
\mu_i\cdot \mu_j = \binom{i+j} i \frac{\omega_{i+j}}{\omega_{i}\omega_{j}} \mu_{i+j}.
$$
\end{corollary}

\noindent {\bf Remark.} The relations \eqref{powers of t} may also be expressed
\begin{equation}
\exp(\pi t) = \sum \omega_i \mu_i, \quad \frac 1{2-t} = \sum \frac{\mu_i}{\alpha_i }.
\end{equation}



We say that a valuation $\varphi\in \Val(\Rn)$ is {\bf monotone} if $\varphi(A) \ge \varphi(B)$ whenever $A\supset B, A,B\in \K^n$; {\bf positive} if $\varphi(A)\ge 0$ for all $A \in \K^n$;  and {\bf Crofton positive} if each of its homogeneous components admits a nonnegative Crofton measure. Such valuations constitute the respective cones $CP \subset M \subset P \subset \Val(\Rn)$.

\begin{exercise} $CP \cap \valson = M\cap \valson = P\cap \valson= \la \mu_0,\dots,\mu_n\ra_+$.
\end{exercise}

In fact the algebra of the vector space of $SO(n)$-invariant valuations on $\Rn$ given in Thm. \ref{structure of valson} is a special case of a more general theorem, due to Alvarez-Fernandes and Bernig, about the {\bf Holmes-Thompson volumes} associated to a smooth {\bf Minkowski space}. Unfortunately, however, there is no ``dual" interpretation as a kinematic formula in this setting.

Recall that a {\bf Finsler metric} on a smooth manifold $M^n$  is a smoothly varying family $g$ of smooth norms on the tangent spaces $T_xM$.
 Let $g_x^*$ denote the dual norm on $T^*_xM$, and $B^*_x := (g_x^{*})^{-1}[0,1]\subset T_x^*M$ the associated field of unit balls. Then for $E \subset M$ the {\bf Holmes-Thompson volume} of $M$ is the measure
\begin{equation}
HT_g(E):=\omega_n^{-1} \int_{\bigcup_{x\in E} B^*_x} \varpi^n
\end{equation}
where $\varpi$ is the natural symplectic form on $T^*M$ (cf. \cite{thompson}).
A {\bf Minkowski space} is the special case of a finite-dimensional normed  vector space $W^n$ with smooth norm $F$. In this case the Holmes-Thompson volume of a subset $S \subset W^n$ may be expressed
\begin{equation}\label{minkowski ht}
HT^F(S) = \omega_n^{-1}\vol_n(S) \vol_n(B_{F^*})
\end{equation}
where $\vol_n$ is the Lebesgue volume of the background euclidean structure.

For a compact smooth submanifold $M^m \subset W$ it is natural to define 
$ HT_m^F(M)$ to be the total Holmes-Thompson volume of $M$ with respect to the Finsler metric on $M$ induced by $F$.
Extending \eqref{minkowski ht} this may also be expressed as
\begin{equation}
HT_m^F(M) = \omega_m^{-1}\int_M \vol_m(B_{(\left.F\right|_{T_xM})^*})\, d\vol_m x
\end{equation}
 where again $\vol_m$ is the $m$-dimensional volume induced by the euclidean structure on $W$. In particular, if $E\in \Gr_m(W)$ and $M\subset E$ then
\begin{equation}\label{def ht vol} 
HT_m^F(M) = \omega_m^{-1} \vol_m((B_F \cap E)^*)\vol_m (M).
\end{equation}
Here the polar is taken as a subset of $E$.

\begin{theorem}[\cite{be05,alv}]
Each $HT_m^F$ extends uniquely to a smooth even valuation of degree $m$ on $W$, which we denote by $\mu_m^F$. Furthermore
 \begin{equation}\label{ht algebra}
 \mu_i^F \cdot \mu_j^F= \binom{i+j} i \frac{\omega_{i+j}}{\omega_i\omega_j}\mu_{i+j}^F.
 \end{equation}
Thus the vector space spanned by the $\mu_i^F$ is in fact a subalgebra of $\Val(W)$, isomorphic to $\R[x]/(x^{n+1})$.
\end{theorem} 
\begin{proof} It is a general fact that for $A \in \K(W), E\in \Gr_m(W)$
\begin{equation}\label{polar ident}
(A\cap E)^* = \pi_E(A^*)
\end{equation}
where the polar on the left is as a subset of $E$, and on the right as a subset of $W$. Observe that if $K\in \K(E^\perp)$ then
\begin{align*}
\frac{(n-m)!}{n!}V(A^*[m],K[n-m])&= \left.\frac{d^m}{dt^m}\right|_{t=0} \vol_n(K + tA^*) \\
&= m! \vol_{n-m}(K) \vol_m(\pi_E(A^*)).
\end{align*}
Hence if we take $A= B_F$ then by \eqref{def ht vol}, \eqref{polar ident} and the characterization of the Fourier transform in terms of Klain functions in Thm. \ref{alesker package} \eqref{af transform} it follows that 
$$
\mu_m^F:= c{V(B_{F^*}[m],\cdot)}\widehat{}
$$
is a smooth even valuation of degree $m$ extending the $m$th Holmes-Thompson volume.

Now by \eqref{def conv} and \eqref{conv 2}
\begin{align*}
\mu_i^F\cdot \mu_j^F &= \widehat{\widehat{\mu_i^F}*\widehat{\mu_j^F}}\\
&= c \left(V(B_{F^*}[i],\cdot)*V(B_{F^*}[j],\cdot)\right)\widehat{}\\
&=c V(B_{F^*}[i+j],\cdot )\widehat{}\\
&= c \mu_{i+j}^F.
\end{align*}
To compute the constant we note that $\mu_i^F= \mu_i$ if $F$ is euclidean, and use the relations \eqref{powers of t}.
\end{proof}

\section{Valuations and integral geometry on isotropic manifolds}
\subsection{Brief definition of valuations on manifolds}
More recently, Alesker \cite{ale06} has introduced a general theory of valuations on manifolds, exploiting the insight that smooth valuations may be applied to any compact subset admitting a normal cycle. Formally, given a smooth oriented manifold $M$ of dimension $n$, the space $\V(M)$ of smooth valuations on the oriented $n$-dimensional manifold $M$ consists of functionals on the space of smooth polyhedra (i.e. smooth submanifolds with corners) $P \subset M$ of the form $P\mapsto \int_P \alpha + \int_{N(P)}\beta$, hence may be identified with a quotient of the space $\Omega^n(M) \times \Omega^{n-1}(SM)$.

\begin{definition}
A {\bf smooth valuation} on $M$ is determined by a pair $(\theta, \varphi) \in \Omega^n(M) \times \Omega^{n-1}(S^*M)$ via
\begin{equation}
\Psi_{\theta,\varphi}(A) := \int_A \theta + \int_{N(A)} \varphi.
\end{equation}
The space of smooth valuations on $M$ is denoted $\V(M)$.
\end{definition}

If $M$ is a vector space then it is natural to consider those smooth valuations determined by translation-invariant differential forms $\omega,\varphi$. 
It follows from Thm. \ref{alesker package} (1) that the subspace of all such smooth valuations coincides with $\valsm(M)$ as defined in section \ref{survey}.

The space $\V(M)$ carries a natural filtration, compatible with the grading on the subspace of translation-invariant valuations if $M$ is a vector space. Parallel to Theorem \ref{alesker package} above is the following.

\begin{theorem}[Alesker \cite{ale05a, ale05b, alefu05, ale05d, ale06}]\label{alesker mfd package}
\begin{enumerate}
\item \label{mfd product item} There is a natural continuous commutative filtered product on $\V(M)$ with multiplicative identity given by the Euler characteristic $\chi$. If $N \subset M$ is an embedded submanifold then the restriction map $r_N:\V(M) \to \V(N)$ is homomorphism of algebras. If $M$ is a vector space then the restriction of this product to translation-invariant valuations coincides with the product of \eqref{product}.
\item If $M$ is compact then the product satisfies Poincar\'e duality, in the sense that the pairing $$(\varphi,\psi):=(\varphi \cdot \psi) (M)$$ is perfect.
\item \label {mfd G} Suppose the Lie group $\barg $ acts transitively on the sphere bundle $SM$.  Put $\V^{\barg}(M)$ for the space of valuations on $M$ invariant under $\barg$. Then $\dim \V^{\barg} (M)<\infty$.
\end{enumerate}
\end{theorem}

The basic idea of the product is that if $X\subset M$ is a ``nice" subset--- say, a piecewise smooth domain--- then the functional $\mu_X:Y\mapsto \chi(X\cap Y)$ is a valuation provided $Y$ is restricted appropriately. This valuation is {\it not} smooth, but a general smooth valuation can be approximated by  linear combinations of valuations of this type. Furthermore it is natural to define the restricted product
$$
\mu_X\cdot \mu_{Y}:= \mu_{X\cap Y}
$$
whenever the intersection is nice enough, and the Alesker product of smooth valuations is the natural extension.
These remarks underlie the following.
\begin{proposition}\label{fu product}
Let $X\subset M$ be a smooth submanifold with corners, and let $\{F_t\}_{t\in P}$ be a smooth proper family of diffeomorphisms of $M$. Thus the induced maps on the cosphere bundle yield a smooth map $\tilde F:P\times S^*M \to S^*M$. Assume that for each $\xi \in S^*M$ the property that the induced map $\tilde F_\xi :P \to S^*M$ is a submersion, and let $dm$ be a smooth measure on the parameter space $P$. Then
\begin{equation}\label{pseudo-kinematic}
[X,\{F_t\}_{t\in P},dm](Y):= \int_P \chi(F_t(X) \cap Y)\, dm
\end{equation}
determines an element of $\V(M)$, and the span of all such valuations is dense in $\V(M)$ in an appropriate sense. Furthermore
\begin{equation}\label{weird product}
\left([X,\{F_t\}_t,dm]\cdot [Y,\{G_s\}_{s\in Q},dn ]\right)(Z) = \int_P \int_Q \chi(F_t(X)\cap G_s(Y) \cap Z)\, dn(s) \, dm(t).
\end{equation}
\end{proposition}
While the right hand side of \eqref{weird product} is well-defined under the given conditions--- this follows from an argument similar to the corresponding part of the proof of Thm. \ref{kf for cm}--- the resulting smooth valuation does not have the same form. An important instance of the construction \eqref{pseudo-kinematic} arises if $P=G$ is a Lie group acting isotropically on $M$ and $dm=dg$ is a Haar measure.

Alesker has observed that the discussion of section \ref{weyl section}, together with the Nash embedding theorem (or, more simply, local smooth isometric embedding of Riemannian manifolds), shows that given any smooth $n$-dimensional Riemannian manifold $M$ there is a canonical {\bf Lipschitz-Killing subalgebra} $ LK(M)\simeq\R[t]/(t^{n+1}) $ of $\V(M)$, obtained by isometrically embedding $M$ in a euclidean space and restricting the resulting intrinsic volumes to $M$.

\subsection{First variation, the Rumin operator and the kernel theorem}
Let $M^n$ be a connected smooth oriented manifold and $\mu \in \V(M)$. Given a vector field $V$ on $M$, denote by $F_t:M\to M$ the flow generated by $V$. We consider the {\bf first variation} of $\mu$ with respect to $V$ given by
\begin{equation}
\delta_V\mu(A):= \left.\frac d{dt}\right|_{t=0} \mu(F_t(A))
\end{equation}
where $A\subset M$ is nice. Clearly $\mu=0$ iff $\delta \mu =0$ and $\mu(\{p\}) = 0$ for some point $p\in M$.

If $\mu=\Psi_{(\theta,\varphi)} \in \Omega^n(M) \times \Omega^{n-1}(S^*M)$ then the first variation may be represented as follows.  Put $\tilde F_t := F_{-t}^*: S^*M\to S^*M$ for the corresponding flow of contact transformations of the cosphere bundle. Then $\tilde F$ is the flow of a vector field $\tilde V$ on $S^*M$. 
Then

\begin{align}
\notag \delta_V\mu(A)&= \left.\frac d{dt}\right|_{t=0}\left[ \int_{F_t(A)} \theta +\int_{N^*(F_t(A))} \varphi\right]\\
\notag &= \left.\frac d{dt}\right|_{t=0}\left[ \int_{F_t(A)} \theta +\int_{\tilde F_{t*}N^*(A))} \varphi\right]\\
\notag &= \left.\frac d{dt}\right|_{t=0}\left[ \int_{A}F_t^* \theta +\int_{N^*(A)} \tilde F_{t}^*\varphi\right]\\
\label{1st variation formula}&= \int_{A}\left.\frac d{dt}\right|_{t=0} F_t^*\theta +\int_{N^*(A)} \left.\frac d{dt}\right|_{t=0}\tilde F_{t}^*\varphi\\
\notag &= \int_{A}\mathcal L_V  \theta +\int_{N^*(A)} \mathcal L_{\tilde V}\varphi\\
\notag &= \int_{A} di_V  \theta +\int_{N^*(A)} (di_{\tilde V}+i_{\tilde V}d)\varphi\\
\notag &=\int_{N^*(A)}i_{\tilde V}(\pi^*\theta + d\varphi)\\
\end{align}
since $\pi_*N(A) = \partial \lcur A\rcur$ and $\partial N(A) =0$. In particular this last expression is independent of the choice of differential forms $\theta,\varphi$ representing $\mu$.

Following Bernig and Br\"ocker \cite{bebr07}, this criterion becomes much clearer and more useful with the introduction of the {\bf Rumin differential} $D$. Recall that if $\alpha$ is a contact form on a contact manifold $M$ then there exists a unique {\bf Reeb vector field} $T$ such that
$$
i_T\alpha\equiv 1,\quad \mathcal L_T\alpha \equiv 0 ,\quad i_Td\alpha \equiv 0.
$$
(Of course these three conditions are redundant.) If $M= S\Rn$ then $T_{(x,v)} = (v,0)$.
Let $Q:= \alpha^\perp$ denote the contact distribution, which carries a natural (up to scale) symplectic structure given by $d\alpha$.

Recall that a  differential form on a contact manifold is said to be {\bf vertical} if it is multiple of the contact form.

\begin{proposition}[\cite{rumin}]\label{def D} Let $S^{2n-1}$ be a contact manifold and $\varphi \in \Omega^{n-1}(S)$. Then there exists a unique vertical form $\alpha \wedge\psi $ such that $d(\varphi +\alpha\wedge\psi) $ is vertical. 
\end{proposition}
\begin{proof}
Observe that given the choice of $\alpha$ there is a natural injection $j$ from sections of $\Lambda^* Q^*$ to $\Omega^* S$, determined by the conditions i) that $j$ followed by the restriction to $Q$ is the identity and ii) that $i_T \circ j =0$. In fact the image of $j$ is precisely the subspace of forms annihilated by $i_T$, as well as the image of $i_T$.

We compute for general $\psi$
\begin{align*}
d(\varphi + \alpha \wedge \psi) &= \alpha \wedge(i_Td(\varphi + \alpha \wedge \psi) ) +
i_T(\alpha \wedge d(\varphi + \alpha \wedge \psi) )\\
&\equiv i_T(\alpha \wedge (d\varphi + d\alpha \wedge \psi)) \quad \mod \alpha.
\end{align*}
Thus $\psi$ will satisfy our requirements iff 
\begin{equation}\label{vertical vanish}
\left.(d\varphi + d\alpha \wedge \psi)\right|_Q \equiv 0
\end{equation}
 But it is not hard to prove the following (actually a consequence of the fact that multiplication by $\omega$ is a Lefschetz operator in an $\sltwo$ structure on $\Lambda^*Q$):
\begin{fact}\label{fact} Suppose $(Q^{2n-2},\omega)$ is a symplectic vector space.
\begin{enumerate} 
\item Multiplication by $\omega$ yields a linear isomorphism $\Lambda^{n-2}Q\to \Lambda^nQ$.
\item Multiplication by $\omega^2$ yields a linear isomorphism $\Lambda^{n-3}Q\to \Lambda^{n+1}Q$.
\end{enumerate}
\end{fact}
The restriction of $d\alpha$ to $Q$ is a symplectic form. Applying this pointwise and using the observations above we find that there is a form $\psi$, uniquely defined modulo $\alpha$, such that \eqref{vertical vanish} holds.
\end{proof}

We define the {\bf Rumin differential} of $\varphi$ to be
$$
D\varphi:= d(\varphi + \alpha \wedge \psi).
$$
It is clear that $D$ annihilates all multiples of $\alpha$ and of $d\alpha$.

\begin{lemma}\label{Lambda Lefschetz}
 $\omega \wedge i_T D\varphi = 0$.
\end{lemma}
\begin{proof} This is equivalent to the relation $d\alpha \wedge D\varphi = 0$. But 
$$
d\alpha \wedge D\varphi = d(\alpha \wedge D\varphi ) = 0
$$
since $\alpha\wedge D\varphi = 0$ by construction.
\end{proof}

Now we can rewrite \eqref{1st variation formula} as 
\begin{align}
\notag\delta_V\mu(A) &= \int_{N^*(A)} i_{\tilde V} (\pi^*\theta + D\varphi)\\
\notag&= \int_{N^*(A)}  \pi^*(i_ V\theta) + i_{\tilde V}(\alpha \wedge i_TD\varphi)\\
\label {1st variation with Rumin}&= \int_{N^*(A)}  \pi^*(i_ V\theta) + i_{\tilde V}(\alpha) \wedge i_TD\varphi\\
\notag&= \int_{N^*(A)}  \pi^*(i_ V\theta) + \langle \alpha,V\rangle \wedge i_TD\varphi
\end{align}
since $N^*(A)$ is Legendrian. Since this expression does not involve derivatives of $V$, this shows that the first variation operator takes values among covector-valued curvature measures. 

\begin{theorem}[Bernig-Br\"ocker \cite{bebr07}] \label{ker thm}
The valuation $\Psi_{\theta,\varphi} =0$ iff $\pi^*\theta + D\varphi =0$ and $\int_{S^*_pM} \varphi =0$ for some point $p\in M$.
\end{theorem}
\begin{proof}
To prove ``if", by finite additivity it is enough to prove that $\Psi_{\theta, \varphi}(A) = 0$ for contractible sets $A$. In this case, if $F_t$ is the flow of a smooth vector field $V$ contracting $A$ to a point $p$, then $N(F_t(A)) \to N(\{p\})=\lcur  S^*_pM\rcur$. Hence by \eqref{1st variation with Rumin}
$$
\Psi_{\theta,\varphi} (A) = \Psi_{\theta,\varphi} (\{p\}) =   \int_{S^*_pM} \varphi =0.
$$

By Proposition \ref{zero curv meas}, to prove the converse it is enough to show that if  $\left.\left(\pi^*(i_ V\theta) + \langle \alpha,V\rangle \wedge i_TD\varphi \right)\right|_Q\equiv 0 \mod d\alpha$ for all $V$ then $\pi^*\theta + D\varphi=0$. But this follows from conclusion (2) of Fact \ref{fact} and Lemma \ref{Lambda Lefschetz}.
%
\end{proof}

Since the normal cycle is closed and Legendrian it is immediate that any multiple of the contact form (i.e. a ``vertical" form with respect to the contact structure of $S^*W$), or any exact form, yields the zero valuation.
\begin{corollary}\label{cor:alt ker} If $M$ is contractible then $\Psi_{0,\varphi} =0$ iff $\varphi$ lies in the linear span of the vertical forms and the exact forms.
\end{corollary}

This characterization also implies conclusion \eqref{mfd G} of Thm. \ref{alesker mfd package}: if $\psi= \Psi_{\theta,\varphi}\in \V^G(M)$ then the kernel theorem implies that $\delta \psi \sim \pi^*\theta + D\varphi\in \Omega^n(S^*M)^G$. But this space is finite dimensional, and (by the kernel theorem again) the kernel of the first variation map $\delta$ is the one-dimensional space $\langle \chi\rangle$.

As another application, observe that in the presence of a riemannian structure on $M$ we may identify a covector-valued curvature measure $\vec\psi$ with a scalar-valued one $\psi$ by putting
$$
\int f \, d\psi^A:= \int \la fn_A,d\vec \psi^A\ra
$$
for smooth domains $A$, where $n_A$ is the outward pointing normal. The first variation may thus be expressed as a scalar-valued curvature measure $\delta \mu$ given by
$$
\int f \, d(\delta \mu)^A = \delta_{fn_A}\mu(A).
$$
Now if $M$ is a euclidean space then we say that a curvature measure $\psi\ge 0$ if the measure $\psi^A \ge 0$ whenever $A\in \K$.

\begin{theorem}\label{monotone characterization}
A valuation $\mu \in \valsm(\Rn)$ is monotone iff $\delta \mu \ge 0$ and $\mu(\{pt\}) \ge 0$. $\square$
\end{theorem}

Since the first variation operator is a graded map of degree $-1$ from $\Val(\Rn)$ to translation-invariant curvature measures we obtain

\begin{corollary}
$\mu\in \Val(\Rn)$ is monotone iff all of its homogeneous components are monotone.
\end{corollary}
\begin{proof} This is easily proved by showing that a translation-invariant curvature measure is non-negative iff its homogeneous components are all non-negative.
\end{proof}

Finally, the Rumin operator also figures prominently in a remarkable formula of Alesker-Bernig for the product of two smooth valuations in terms of differential forms representing them. We will only need the following general characterization.
\begin{theorem}[Alesker-Bernig \cite{ale-be09}]\label{ab product formula}
The product of two smooth valuations $\Psi_{\theta_i,\varphi_i}, i =1,2$ may be expressed as $\Psi_{\theta_0,\varphi_0}$, where ${\theta_0,\varphi_0}$ may be computed from the ${\theta_i,\varphi_i}$ using only the Rumin differential $D$ and canonical fiber integrals. $\square$
\end{theorem}

It is interesting to note that the expression for $\Psi_{\theta_0,\varphi_0}$ given in \cite{ale-be09} is not symmetric in $i=1,2$. In fact it gives a proof of the following observation of Bernig:

\begin{corollary}\label{cor:module} The space of smooth curvature measures is a module over the algebra of smooth valuations. 
\end{corollary}

\subsection{The filtration and the transfer principle for valuations} \label{val transfer section} 
Alesker \cite{ale05d} showed that the algebra $\V(M^n)$ admits a natural filtration $\V(M) =\V_0(M)\supset \dots \supset \V_n(M)$ that respects the product. The filtration index of a valuation $\mu$ may be expressed as the largest value of $i$ such that $\mu$ may be represented by  differential forms $\theta, \varphi$ where in some family of local trivializations of $S^*M$ all terms of the differential form $\varphi$ involve at least $i$ variables from the coordinates of the base space.
Thus $\V_n(M)=\{\Psi_{\theta,0}:\theta \in \Omega^n(M)\}$ is the space of smooth signed measures on $M$. The associated graded algebra $gr(\V(M))=\bigoplus_{i=0}^n \V_{i}(M)/\V_{i+1}(M)$ is naturally isomorphic to the algebra $\Gamma(\valsm(TM))$ of sections of the infinite-dimensional vector bundle $\valsm(T M) \to M$.

A valuation $\varphi \in \V_k(M)$ may be thought of as defining a ``$k$-dimensional volume" of $M$: for $N^k\subset M$ 
\begin{equation}\label{reduced val}
\varphi(N^k) = \int_N \kl_{[\varphi]_x} (T_xN) \, d\vol_kx
\end{equation}
where $[\varphi] \in \Gamma(\valsm(TM))$ corresponds to the equivalence class of $\varphi$ in $ \V_k(M) /\V_{k+1}(M)$. Reducing further, this value in fact only depends on the even parts of the $[\varphi]_x\in \valsm_k(T_xM)$. In this formula $\varphi$ may be thought of as operating on any compact $C^1$ submanifold of $M$, or more generally on any appropriately rectifiable set of the same dimension.

%


If $G$ acts isotropically on $M$ then this gives rise to a transfer principle at the level of first order formulas, similar to \eqref{1st order}, as follows. The isomorphism above gives rise to the restriction
$$
gr\V^{G}(M)\simeq\Gamma(\valsm(TM))^G \simeq \Val^{H}(T_{o}M)
$$ 
where $H\subset G$ is the subgroup stabilizing the representative point $o$. Thus there is a canonical map $\V^G(M) \to \Val^H(T_oM)$; for $\varphi \in \V^G(M)$ we put $[\varphi] \in \Val^H(T_oM)$ for its image under this map. If $\psi \in\V_{k+1}$ then $\psi(V)=0$ whenever $V \subset M$ is a smooth submanifold of dimension $ k$. Thus if $[\varphi] \in \Val_k^H(T_oM)$ then $[\varphi](V) = \varphi(V)$ is well-defined, and in fact may be expressed
\begin{equation}\label{eq:reduced eval}
[\varphi](V) = \int_V \kl_{[\varphi]}(T_xV) \, dx
\end{equation}
where $T_xV \subset T_xM$ is identified via the group action with a subspace of $T_oM$.

  Obviously $H$ acts transitively on the sphere of $T_oM$. Thus   $\Val^H \subset \valsm_+$ by Prop. \ref{valg subset val+}, and we may consider the kinematic operator
\begin{equation}
k_H: \Val^H(T_{o}M)\to \Val^H(T_{o}M)\otimes\Val^H(T_{o}M).
\end{equation}
Furthermore, any element $\phi \in \Val^H_k(T_oM)$ acts on $k$-dimensional submanifolds $V^k$ by

\begin{theorem}\label{val transfer thm} In the setting above, let $V^k,W^l \subset M$ be compact $C^1$ submanifolds, and let $\varphi \in \V_{k+l-n}^G(M) $ correspond to $[\varphi]\in \Val^H_{k+l-n,+}(T_{o}M)$ as above. Then
\begin{equation}
\int_G \varphi(V \cap g W)\, dg = k_H([\varphi])(V,W).
\end{equation}
\end{theorem}
Here the elements of  $\Val^H(T_oM)$ are evaluated on $V,W$ as in the discussion above, noting that since $k_H$ is a graded operator it is only necessary to evaluate valuations of degree $k,l$ on $V,W$ respectively.

In other words, at the level of first order formulas the kinematic formulas for $(M,G)$ and $(T_oM, \bar H)$ are identical. Since the space of invariant valuations on a vector space is graded, and the Klain map on each graded component is injective, the first order formulas there carry the same information as the full kinematic operator $k_H$. However, the spaces $\vgm$ are only filtered, so one loses information 
 in passing from the full kinematic operator $k_G$ to the first order version. 

%


\subsection{The ftaig for compact isotropic spaces}
\begin{lemma}\label{lem:intersection product formula}
\begin{equation*}\label{mult}
\mu^G_A \cdot \varphi = \int_G \varphi(gA\cap \cdot)\, dg.
\end{equation*}
\end{lemma}
\begin{proof} Restricting to valuations of the form considered there, this follows from Proposition \ref{fu product}
\end{proof}

\begin{theorem}[\cite{fu06, befu06, ale-be09}] \label{compact ftaig}
 Let $(M,G)$ be a compact isotropic Riemannian manifold. Put $\vgmstar$ for the the dual space to the finite dimensional vector space $\vgm$ and
$p: \vgm\to {\vgmstar}$ for the Poincar\'e duality map of Thm. \ref{alesker mfd package}; put $m_G^*: {\vgmstar}  \to{\vgmstar}\otimes {\vgmstar }$ for the adjoint of the restricted Alesker product map.
 Then the following diagram commutes:
 \begin{equation}\label{fundamental identity}
 \begin{CD}
\vgm & @>{\vol(M)^{-1}k_{G}}>> & \vgm\otimes \vgm  \\
@V{p}VV & & @V{p\otimes p}VV \\
\vgmstar&@>{m_G^*}>>  & \vgmstar\otimes \vgmstar
 \end{CD}
\end{equation}
\end{theorem}
\begin{proof}
 Given a smooth polyhedron $A\subset M$, define $\mu^G_A \in \vgm$ by $\mu^G_A:= \int_G \mu(gA)\, dg$, i.e. the invariant valuation $\mu^G_A (B) := \int_G \chi(gA\cap B)\,dg$, where $dg$ is the Haar measure with the usual normalization. As in the translation-invariant case, these valuations span $\vgm$.
 Using Lemma \ref{fu product},  to prove \eqref{fundamental identity} it is enough to show that
\begin{equation}\label{main}
\left((p\otimes p)\circ k_G(\mu^G_A)\right)\left(\mu^G_B \otimes \mu^G_C\right)=\vol(M) \left(m_G^*\circ p (\mu^G_A)\right)\left(\mu^G_B \otimes \mu^G_C\right)
\end{equation}
for all polyhedra $A,B,C$. 

Direct calculation from \eqref{mult} implies that that $p(\mu^G_A)  =  \vol(M) ev_A$, where $ev_A$ is the evaluation at $A$ functional on $\vgm$. Using the fact that $p$ is self-adjoint it is now easy to compute that both sides of \eqref{main} are equal to $\vol(M)^2(  \mu^G_A\cdot\mu^G_B \cdot \mu^G_C )(M)$. 
\end{proof}

\subsection{Analytic continuation}\label{analytic continuation section} As in section \ref{transfer section}, some isotropic spaces come in families, indexed both by dimension and by curvature. The main examples are the real space forms (spheres, euclidean spaces, and hyperbolic spaces), the complex space forms ($\CP^n$, $\Cn$ under the action of the holomorphic euclidean group $\barun$, and $\C H^n$) and the quaternionic space forms ($\HH P^n$, $\HH^n$ under the action of the quaternionic euclidean group $Sp(n) \times Sp(1)$, and $\HH H^n$). There are also octonian versions of these families in octonian dimensions 1 and 2, the 2-dimensional projective case being the Cayley plane.

We already have one method--- the transfer principle--- for comparing the integral geometry within each of these families. Another method arises from a more geometric perspective, viewing the curvature $\lambda$  as a parameter in the various kinematic and product formulas. Fortunately all of the formulas are analytic in $\lambda$, so we may extend results from the compact $\lambda>0$ cases to the noncompact cases $\lambda\le 0$, where direct calculations are more difficult.

The model case is that of the real space forms $M_\lambda$. The transfer principle shows that the kinematic formulas are the same at the level of curvature measures if the latter are identified via their correspondence with invariant differential forms. In the cases $\lambda=0,1$ we have seen in Thms. \ref{ftaig} and \ref{compact ftaig} that the kinematic formulas are related to the Alesker product via Poincar\'e duality, although a priori this duality has different meanings in the two cases. Clearly the same is true for all $\lambda\ge 0$, 

What about the cases $\lambda<0$? On the face of it there can be no ftaig in the hyperbolic case: the invariant valuations are obviously not compactly supported, and the natural Poincar\'e duality \cite{ale05d} for smooth valuations on a manifold pairs a general valuation with a valuation with compact support. Nevertheless such a pairing does exist for {\it invariant} valuations: let $\tau_i \in \V^{G_\lambda}(M_\lambda)$ denote the invariant valuation corresponding to $t^i \in \V^{G_0}(\Rn) = \valson(\Rn)$ under the transfer principle, $i=0,\dots,n$ (this is possible since in this case the correspondence between invariant curvature measures and invariant valuations is bijective). Let $\{\tau_i^*\}$ denote the dual basis. Rescaling as necessary in the cases $\lambda>0$, the ftaig may be restated in a unified way for all $\lambda\ge 0$ by taking the Poincar\'e duality map to be
$$
\langle p(\phi),\psi\rangle := \tau_n^*(\phi \cdot \psi)
$$
for $\phi,\psi \in \V^{G_\lambda}(M_\lambda), \ \lambda\ge 0$.

Now analytic continuation shows that the same form of the ftaig holds also for $\lambda<0$. To carry this out we construct a common model for the space of invariant valuations on all of the $M_\lambda$, wherein the Alesker products, kinematic operators and Poincar\'e pairings vary polynomially with $\lambda$. The Lie algebras $\frak g_\lambda$ of \eqref{lie algebras} may be represented by a family of Lie brackets on the common vector space $\Rn \oplus \frak{so}(n)$ by putting
$$
[(v,h),(w,j)]_\lambda:= (hw -jv, [h,j]+ \lambda(v\otimes w - w \otimes v)).
$$
The tangent spaces $T_{\bar o} SM_\lambda$ are represented by $\Rn \times \frak{so}(n)/\frak{so}(n-1) \simeq \Rn\oplus \R^{n-1}$, and in these terms the transfer map of Thm. \ref{transfer} is the identity. In particular, the pullback of the contact form to $T_{id}G_\lambda= \frak g_\lambda$ is $\alpha(v,h)= v_1$, independent of $\lambda$. Therefore the Maurer-Cartan equations imply that the exterior derivatives $d_\lambda$ on $\Omega^{*G_\lambda}(SM_\lambda)$ are represented by a polynomial family of operators $\Lambda^{*}(\Rn\oplus\R^{n-1})^{SO(n-1)}\to \Lambda^{*+1}(\Rn\oplus\R^{n-1})^{SO(n-1)}$. Furthermore the uniqueness statement in Prop. \ref{def D} implies that the same is true of the Rumin derivatives $D_\lambda$.  Thm. \ref{ab product formula} now implies that the Alesker product also varies polynomially, which completes the proof.

The discussion above is simplified somewhat by the fact that the map from invariant curvature measures to invariant valuations is injective for the real space forms. Nevertheless it extends also to the 1-parameter families of complex, quaternionic and octonionic space forms via the following general statement.

\begin{theorem} Let $G_\lambda$ be a 1-parameter family of Lie groups acting isotropically on spaces $M_\lambda=G_\lambda$, with the following properties:
\begin{enumerate}
\item\label{common model 1} the associated Lie algebras $\frak g_\lambda$ are given by a family $[\,,\,]_\lambda$ of Lie brackets, depending analytically on the parameter $\lambda$, on a fixed finite-dimensional real vector space $\frak g$;
\item\label{common model 2} the $G_\lambda$ have a common subgroup $H$ such that $M_\lambda = G_\lambda/H$, and the restriction of $[\,,\,]_\lambda$ to the associated Lie algebra $\frak h\subset \frak g$ is constant;
\item\label{constant adjoint} the adjoint action of $\frak h$ on $\frak g_\lambda$ is independent of $\lambda$;
\item \label{spaces compact}if $\lambda>0 $ then $M_\lambda$ is compact.
\end{enumerate}
Then there exists a family of elements $f_\lambda \in \V^{G_\lambda*}(M_\lambda)$ determined by the condition
\begin{equation}\label{fundamental homomorphism}
f_\lambda(\Psi_{\varphi, c\, d\vol_{M_\lambda}}) \equiv c
\end{equation}
whenever $\varphi \in \Omega^{n-1, G_\lambda}(SM_\lambda)$, such that, putting 
\begin{equation}\label{vary pd}
p=p_\lambda:\V^{G_\lambda}(M_\lambda)\to \V^{G_\lambda*}(M_\lambda) \quad \text{by}\quad \langle p(\phi),\psi\rangle:= f_\lambda(\phi\cdot \psi)
\end{equation}
and $M= M_\lambda, G=G_\lambda$, the diagram \eqref{fundamental identity} commutes for all $\lambda$.
\end{theorem}
\begin{proof} Since $T_oM_\lambda \simeq \frak g/\frak h$ for the distinguished point $o=o_\lambda \in M_\lambda$ (i.e. the image of the identity element of $G_\lambda$), condition \eqref{constant adjoint} implies that the identity map of $\frak g$ induces isometries \eqref{transfer hypothesis} between these spaces, intertwining the action of $H$, for the various values of $\lambda$. The proof of Thm. \ref{transfer} implies that the canonical 1-forms of the $SM_\lambda$, and their differentials, correspond to fixed elements $\alpha\in \Lambda^1(\frak g/\frak k)^K, \omega \in\Lambda^2(\frak g/\frak k)^K$ respectively, independent of $\lambda$. Furthermore Thm. \ref{transfer} implies that the identification $\Lambda_0:=\Lambda^n(\frak g/\frak h)^H \oplus \left(\Lambda^{n-1}(\frak g/\frak k)^K\right)/(\alpha,\omega)\simeq\curv^{G_\lambda}(M_\lambda) $ is independent of $\lambda$, as is the coproduct structure on $\Lambda_0$ induced by the  $\tilde k_\lambda= \tilde k_{G_\lambda}$. Here as usual $K\subset H$ is the isotropy subgroup of the distinguished point $\bar o=\bar o_\lambda \in SM_\lambda$ and $\frak k \subset \frak h$ is its Lie algebra.

By the reasoning above, the Rumin operators of the $SM_\lambda$ induce an analytic family of operators $D_\lambda: \Lambda^{n-1}(\frak g/\frak k)^K/(\alpha,\omega) \to \Lambda^n(\frak g/\frak k)^K$, and by the Kernel Theorem \ref{ker thm} the kernel of the map $\Psi_\lambda:\Lambda_0 \to \V^{G_\lambda}(M_\lambda)$ is the same as that of the map
$$\Delta_\lambda:\Lambda_0 \to \Lambda_1:= \Lambda^{n-1}(\frak h/\frak k)^K\oplus \Lambda^n(\frak g/\frak k)^K
$$
given by
$$
\Delta_\lambda(\beta,\gamma):=(r(\gamma), \pi^*\beta + D_\lambda \gamma) 
$$
where $r:\Lambda^*(\frak g/\frak k)\to \Lambda^*(\frak h/\frak k)$ is the restriction map.

\newcommand\vgamlam{\V^{G_\lambda}(M_\lambda)}

Now choose a subspace $W \subset \Lambda_0$ such that the restriction of $\Psi_\lambda $ gives a linear isomorphism  $\Psi_\lambda':W \to \vgamlam$ locally in $\lambda$, and consider 
 the family $\{\tilde m_\lambda\}$ of compositions 
$$
\begin{CD}
W\otimes W & @>{\Psi_\lambda'\otimes\Psi_\lambda'}>>& \V^{G_\lambda}(M_\lambda)\otimes \V^{G_\lambda}(M_\lambda)& @>{m_\lambda}>> & \V^{G_\lambda}(M_\lambda) & @>{\Psi_\lambda'^{-1}}>> & W
\end{CD}
$$
We claim that this family is analytic in $\lambda$. To see this we observe that $\tilde m_\lambda$ can also be expressed as the composition
$$
\begin{CD}
W\otimes W & @>i>>&\Lambda_0\otimes \Lambda_0& @>{\bar m_\lambda}>> &\Lambda_0 = \ker_\lambda \oplus W & @>\pi_\lambda>> & W
\end{CD}
$$
where $i$ is the inclusion, $\bar m_\lambda$ is the map induced by the construction of Thm. \ref{ab product formula}, $\ker_\lambda$ is the kernel of $\Psi_\lambda$ and $\pi_\lambda$ is the projection with respect to the given decomposition. As in the discussion above, $\bar m_\lambda$ is analytic in $\lambda$, so it remains to show that the same is true of $\pi_\lambda$. But $\ker_\lambda$ is also the kernel of $\Delta_\lambda:\Lambda_0 \to \Lambda_1$. Thus we may write
$\Lambda_1 = U \oplus V$ 
where the composition of the $\left.\Delta_\lambda\right|_W$ with the projection to $U$ is an isomorphism locally in $\lambda$, which is clearly analytic in $\lambda$. Hence the same is true of the inverse of this composition, and $\pi_\lambda$ is the composition of the projection to $U$ with this inverse, which is thereby again analytic.

The proof is concluded by applying analytic continuation, noting that by condition \eqref{spaces compact} the Poincar\'e duality maps $p_\lambda$ have the given form \eqref{fundamental homomorphism}, \eqref{vary pd} for $\lambda >0$.
\end{proof}


Another sticky technical point concerns the extension of the Crofton formulas \eqref{affine general crofton} to the hyperbolic case, where the spaces of totally geodesic submanifolds of various dimensions  replace the affine Grassmannians. In the spherical cases the corresponding formulas are literally special cases of the kinematic formula where the moving body is taken to be a totally geodesic sphere.
In the euclidean case these formulas may be regarded as limiting cases of kinematic formulas where the moving body is taken to be a disk of the given dimension and of radius $R\to \infty$: for finite $R$ there are boundary terms but these become vanishingly small in comparison with the volume in the limit. However, neither of these devices are available in the hyperbolic cases. Nevertheless the formulas may be analytically continued:

\begin{theorem}\label{transfer crofton} Let $\barGr_\lambda(n,k)$ denote the Grassmannian of all totally geodesic submanifolds of dimension k in $M^n_\lambda$. If we put 
$$
\phi_\lambda (A):= \int_{\barGr_\lambda(n,n-1)} \chi(A\cap \bar P) \, dm_\lambda^{n-1}(\bar P),
$$
where $dm_\lambda^k$ is the Haar measure on $\barGr_\lambda(n,k)$, 
then $\phi_\lambda\in \V^{G^n_\lambda}(M^n_\lambda)$. Furthermore if these measures are normalized appropriately then
\begin{equation}
\phi_\lambda^k =\int_{\barGr_\lambda(n,n-k)} \chi(\cdot \cap \bar Q) \, dm_\lambda^k(\bar Q)=
\tau_k + \sum_{i>k}^n p^k_i(\lambda) \tau_i
\end{equation}
where each $p^k_i$ is a polynomial.
\end{theorem}
\begin{proof} 
Consider the pullback to $\Rn$ of the standard Riemannian metric on the sphere $S^{n}_R$ of radius $R$ under the spherical projection map
$$
\Rn\to S^{n}_R,\quad p \mapsto R\frac{(p,R)}{\sqrt{|p|^2 + R^2}}
$$
In terms of polar coordinates $(r,v)\mapsto rv,\ \R_+ \times S^{n-1} \to \Rn$, one computes the pullback metric to be
$$
\frac {dr^2}{(1+ \frac{r^2}{R^2})^2} + \frac {r^2 dv^2}{1+ \frac{r^2}{R^2}}= \frac {dr^2}{(1+\lambda{r^2}{})^2} + \frac {r^2 dv^2}{1+ \lambda{r^2}{}}
$$
where $dv^2$ is the usual metric on the unit sphere. (Of course these coordinates cover only the northern hemisphere.)
Meanwhile the Hilbert metric on $B(0,R)$, which yields the Klein model for $H^n$, with curvature $-1$, may be expressed
\begin{equation}
ds_{Hilbert}^2 = \frac {dr^2}{R^2(1- \frac{r^2}{R^2})^2} + \frac {r^2 dv^2}{R^2(1- \frac{r^2}{R^2})}.
\end{equation}
Scaling this up by a factor of $R^2$ yields
\begin{equation}
\frac {dr^2}{(1- \frac{r^2}{R^2})^2} + \frac {r^2 dv^2}{(1- \frac{r^2}{R^2})} = \frac {dr^2}{(1+\lambda{r^2}{})^2} + \frac {r^2 dv^2}{1+ \lambda{r^2}{}} =: ds_\lambda^2
\end{equation}
where again $\lambda=-R^{-2}$ is the curvature.

In these coordinates all of the various totally geodesic hypersurfaces appear as hyperplanes (or their intersections with the disk) of $\Rn$. In other words the codimension 1 affine Grassmannian $\barGr:= \barGr_{n-1}(\Rn)$ of $\Rn$ is a model for all of the various affine Grassmannians as the curvature $\lambda$ varies. We again parametrize $\barGr$ by polar coordinates, i.e. put $\bar P(r,v):=\{x:v\cdot x= r\}$. Then
\begin{equation}
dm_\lambda= \frac {dr \, dv}{\left(1+\lambda r^2\right)^{\frac{n+1} 2}}
\end{equation}
This is easily checked for $\lambda\ge0$, and is a direct calculation if $\lambda<0$ and $n=2$. If $\lambda<0$ note that this measure assigns the value 0 to any set of hyperplanes that does not meet the relevant disk.
\end{proof}

The complex space forms also fit together nicely as a one-parameter family. Adapting well known formulas for the Fubini-Study metric (\cite{ko-no}, p. 169) we may define the Hermitian metrics
$$
ds_{\C,\lambda}^2:= \frac{(1+\lambda|z|^2)ds_{\C}^2 -\lambda(\sum z_i d\bar z_i)(\sum \bar z_i d z_i)}{(1+ \lambda |z|^2)^{2}}
$$
for all $\lambda \in \R$, defined on all of $\Cn$ if $\lambda \ge 0$ and on the ball $B(0, \frac 1{\sqrt{-\lambda}})\subset \Cn$ if $\lambda <0$, where $ds_{\C}^2$ denotes the standard Hermitian metric on $\Cn$. This yields a metric of constant holomorphic sectional curvature $4\lambda$ on an affine coordinate patch (covering the whole space if $\lambda \le 0$) of the corresponding complex space form.

\begin{exercise} Complete this discussion by considering the invariant measures on the spaces of totally geodesic complex submanifolds, and also on totally geodesic {\it isotropic} submanifolds, i.e. submanifolds on which the K\"ahler form vanishes. These are the unitary orbits of $\R^k, \R P^k,  \R H^k, 0\le k\le n$, in $\Cn, \C P^n, \C H^n$ respectively.
\end{exercise}

\begin{exercise} Do the same for the quaternionic space forms.
\end{exercise}

\subsection{Integral geometry of real space forms}\label{ig of real}

We use the transfer principle and analytic continuation to study the isotropic spaces 
$$(\R^n,\barson),\ (S^n,SO(n+1)),\ (H^n, SO(n,1)).
$$
For simply connected spaces $M^n_\lambda$ of constant curvature $\lambda$, put $\tau_i$ for the valuations corresponding to $t^i$ under the transfer principle. These are {\it not} the Lipschitz-Killing valuations, but they {\it do} correspond to appropriate multiples of the elementary symmetric functions of the principal curvatures for hypersurfaces. Nijenhuis's theorem and the transfer principle yield
 \begin{equation}\label{pkf transfer}
k_{M^n_\lambda} (\tau_l) = \frac{\alpha_n}{2^{n+1}} \sum_{i+j = n + l} \tau_i \otimes \tau_j, \quad l= 0,\dots, n.
\end{equation}
Taking $\lambda >0$, the values of the $\tau_i$ on the totally geodesic spheres $S^j(\lambda) \subset S^n(\lambda) = M^n_\lambda$ are
\begin{equation*}
\tau_i(S^j(\lambda)) = \delta^i_j \cdot 2 \left(\frac 2 {\sqrt \lambda}\right)^i.
\end{equation*}
Therefore
\begin{equation}\label{chi formula}
\chi = \sum_{i =0}^ {\lfloor \frac n 2\rfloor} \left(\frac \lambda 4\right)^i \tau_{2i}
\end{equation}
in $M^n(\lambda)$, $\lambda >0$, and by analytic continuation this formula holds for $\lambda\le 0$ as well. 

Therefore the principal kinematic formula in $M^n_\lambda$ may be written
\begin{equation}
 k_{M^n_\lambda} (\chi) = \sum_{i =0}^ {\lfloor \frac n 2\rfloor} \left(\frac \lambda 4\right)^i k_{M^n_\lambda}(\tau_{2i})
\label{pkf space forms}= \frac{\alpha_n}{2^{n+1}} \sum_{i} \tau_i \otimes \left(\sum_{j=0}^\infty \left(\frac \lambda 4\right)^j \tau_{n-i + 2j}\right).
\end{equation}

Returning to the case $\lambda >0$, let $\phi$ denote the invariant valuation given by $\frac 2{\sqrt{\lambda}}$ times the average Euler characteristic of the intersection with a totally geodesic hypersphere. Then, using the standard normalization \eqref{dg normalization} for the Haar measure on $SO(n+1)$,
\begin{align}
\notag \phi^k &:= \left(\frac 2{\sqrt{\lambda}}\right)^k\frac {\lambda^{\frac n 2}}{\alpha_n } \int_{SO(n+1)} \chi(\cdot \cap g S^{n-k}(\lambda)) \, dg \\
\label{powers of phi}&= \frac {2^k\lambda^{\frac{n-k} 2}} {\alpha_n} k_{S^n(\lambda)} (\chi)(\cdot, S^{n-k}(\lambda)) \\
\notag&=   \sum_{j=0}^\infty \left(\frac \lambda 4\right)^j\tau_{2j+k},\quad k = 0,\dots,n.
\end{align}
In particular, if $N^k$ is a compact piece of a totally geodesic submanifold of dimension $k$, then
$$
\phi^k(N) = \tau_k(N) = \frac{2^{k+1}}{\alpha_k} |N|.
$$

Comparing \eqref{chi formula} and \eqref{powers of phi} we find that
\begin{equation}\label{chi tau phi}
\chi = \tau_0 + \frac \lambda 4 \phi^2
\end{equation}
or
\begin{equation}
\tau_0 = \chi - \frac \lambda 4 \int_{\Gr_{\lambda,n-2}} \chi(\cdot \cap H) \, dH.
\end{equation}
This formula is due to Teufel \cite{teufel} and Solanes \cite{solanes}.

Now the principal kinematic formula \eqref{pkf space forms} may be written
\begin{equation}\label{pkf, 1st phi form}
k_{M_\lambda^n}(\chi) =\frac{\alpha_n}{2^{n+1}} \sum_{i + j = n}  \tau_i \otimes \phi^j .
\end{equation}
By the multiplicative property,
\begin{equation}\label{pkf, phi form}
k_{M_\lambda^n}(\tau_k) =\frac{\alpha_n}{2^{n+1}} \sum_{i + j = n}  \tau_i \otimes \left(\phi^j\cdot \tau_k\right).
\end{equation} 
Comparing with \eqref{pkf transfer} we find that
\begin{equation}\label{product phi tau}
\phi^j \cdot \tau_i=\tau_{i+j} 
\end{equation}
This is the {\it reproductive property} of the $\tau_i$.
At this point we may write
\begin{equation}
k_{M^n_\lambda}(\psi) = (\psi \otimes \tau_0) \cdot \left(\sum_{i+j = n} \phi^i \otimes \phi^j \right)
\end{equation}
whenever $\psi$ is an invariant valuation on $M^n_\lambda$. By analytic continuation these formulas are valid for $\lambda\le 0$ as well, where $\phi$ is the corresponding integral over the space of totally geodesic hyperplanes, normalized so that $\phi(\gamma) = \frac 2\pi |\gamma|$ for curves $\gamma$.

By \eqref{powers of phi}, the valuations  $\phi, \tau_1$ and the generator $t$ of the Lipschitz-Killing algebra all coincide if the curvature $\lambda=0$. For general $\lambda$, the relations \eqref{chi tau phi} and \eqref{product phi tau} yield
\begin{equation}
\tau_i = \phi^i- \frac \lambda 4 \phi^{i+2}.
\end{equation}

There is also a simple general relation between $\phi$ and $t$. To put it in context, denote by $V^n_\lambda$ the algebra of invariant valuations on $M^n_\lambda$. Since each $M^n_\lambda$ embeds essentially uniquely into $M^{n+1}_\lambda$, and every isometry of $M^n_\lambda$ extends to an isometry on $M^{n+1}_\lambda$, there is a natural surjective restriction homomorphism $V^{n+1}_\lambda \to V^{n}_\lambda$. Put $V^\infty_\lambda $ for the inverse limit of this system. Thus $V^\infty_\lambda$ is isomorphic to the field of formal power series in one variable, which may be taken to be either $\phi$ or $t$. The valuations $\tau_i$ may also be regarded as living in $V^\infty_\lambda$ since they behave well under the restriction maps. The relations among valuations given above are also valid in $V^\infty_\lambda$ except for those that depend explicitly on the kinematic operators. These depend on the dimension $n$ and definitely do not lift to $V^\infty_\lambda$. 

\begin{proposition}
\begin{equation}\label{t phi}
t = \frac{\phi}{\sqrt{1-\frac{\lambda\phi^2} 4}}, \quad \phi = \frac{t}{\sqrt{1+\frac{\lambda t^2} 4}}
\end{equation}
\end{proposition}
\begin{proof} 
We use the template method to prove that
\begin{equation}\label{t^2 expansion}
t^2 = \frac {\phi^2}{1-\frac{\lambda \phi^2}{4}}= \phi^2\left(1+ \frac{\lambda } 4\phi^2 + \left(\frac{\lambda } 4\right)^2\phi^4+\dots \right)
\end{equation}
for $\lambda >0$, so $t$ must be given by the square root that assigns positive values (lengths) to curves. The generalization to all $\lambda \in \R$ follows by analytic continuation.

Since everything scales correctly it is enough to check this for $\lambda = 1$. In fact it is enough to check that that the values of the right- and left-hand sides of \eqref{t^2 expansion} agree on spheres $S^{2l}$ of even dimension. Clearly 
$$
\phi^{2k}(S^{2l}) = \begin{cases}2\cdot 4 ^k,  & k\le l \\
0,& k>l
\end{cases}
$$
so the right-hand side yields ${8l}$.

To evaluate the left-hand at $S^{2l}$ we recall that 
$$
t^2(S^{2l}) = 2t^2(B^{2l+1}) = \frac{4\omega_2}{\pi^2} \mu_2(B^{2l+1}) = \frac 4 \pi \mu_2(B^{2l+1}).
$$
Meanwhile, Theorem \ref{steiner's formula} for the volume of the $r$-tube about $B^{2l+1}$ yields
\begin{equation*}
\omega_{2l+1}(1 + r)^{2l+1}= \sum \omega_{2l-i+1} \mu_i (B^{2l+1}) r^{2l+1-i}.
\end{equation*}
Equating the coefficients of $r^{2l-1}$ we obtain
$$
\mu_2(B^{2l+1}) = \frac{\omega_{2l+1}}{\omega_{2l-1}} \binom{2l+1} 2 = {2\pi l}
$$
in view of the identity
$$
\frac{\omega_n}{\omega_{n-2}} = \frac {2\pi}{n}.
$$
\end{proof}

\section{Hermitian integral geometry} Our next goal is to work out the kinematic formulas for $\Cn$ under the action of the unitary group $U(n)$.
The starting point is Alesker's calculation of the Betti numbers and generators of $\valun(\Cn)$, which we state in a slightly stronger form. Let $D^n_{\C}$ denote the unit disk in $\C^n$.

\begin{theorem}[\cite{ale03b}]\label{dim count}
$\valun(\C^n)$ is generated by $t \in \Val_1^{U(n)}$ and 
\begin{equation}
s:= \int_{\barGr^{\C}_{n-1}} \chi (\cdot \, \cap \bar P)\, d\bar P \in \Val_2^{U(n)}
\end{equation}
where ${\barGr^{\C}_{n-1}}$ is the codimension 1 complex affine Grassmannian and the Haar measure $d \bar P$ is normalized so that 
$$
s(D^1_{\C}) = 1.
$$
There are no relations between $s,t$ in (weighted) degrees $\le n$.
\end{theorem}

By Alesker-Poincar\'e duality $\dim \Val_{2n-k}^{U(n)} = \dim \Val_{k}^{U(n)}$ so this determines the Betti numbers completely. In fact, the Poincar\'e series for $\valun$ is
\begin{equation}\label{poincare series}
\sum \dim \Val_i^{U(n)} x^i = \frac{(1-x^{n+1})(1-x^{n+2})}{(1-x)(1-x^2)}=:\sigma(x).
\end{equation}
This follows at once from the palindromic nature of $\sigma(x)$, which may be expressed as the identity
$$
x^{2n}\sigma(x^{-1}) = \sigma(x).
$$ 

By \eqref{product} and Theorem \ref{structure of valson}, 
$$
s^kt^l(A) = c\int_{\bar\Gr^{\C}_{n-k}} \mu_l(A\cap \bar P)\, d\bar P
$$
for some constant $c$.

\begin{exercise}
Show that 
\begin{equation}
\widehat{(s^kt^l)} = c s^{n-k} * t^{2n-l}
\end{equation}
and\begin{equation}
\widehat{(s^kt^l)} (A) = c\int_{\Gr^{\C}_{k}} \mu_{2n-2k-l}(\pi_E(A)) \, dE, \quad A \in \K(\C^n).
\end{equation}
for some constants $c$.
\end{exercise}

In fact Alesker \cite{ale03b} showed that $\{s^kt^l: k\le \min ( k + \frac l 2,n -k-\frac l 2)\}$, is a basis for $\valun$, as is its Fourier transform. 

\subsection{Algebra structure of $\valun(\Cn)$}\label{sect:algebra structure} The next step in unwinding the integral geometry of the complex space forms is
 \begin{theorem}[\cite{fu06}]\label{valun structure} Let $s,t$ be variables with formal degrees 2 and 1 respectively. Then the graded algebra $\valun(\Cn)$ is isomorphic to the quotient $\R[s,t]/(f_{n+1}, f_{n+2})$, where $f_k$ is the sum of the terms of weighted degree $k$ in the formal power series
$\log (1 + s + t)$. The $f_k$ satisfy the relations
\begin{align}
\notag f_1 & = t\\
\notag f_2 & = s-\frac{t^2}{2}\\
\label{recursion}
ksf_k+(k+1)&tf_{k+1}+(k+2)f_{k+2}   =0, \quad k \geq 1.
\end{align} 
\end{theorem}
\begin{proof} For general algebraic reasons (viz. \eqref{poincare series} and Alesker-Poincar\'e duality), the ideal of relations is generated by independent homogeneous generators in degrees $n+1,n+2$. Since the natural restriction map $\Val^{U(n+1)}(\C^{n+1}) \to \valun(\C^n)$ is an algebra homomorphism, and maps $s$ to $s$ and $t$ to $t$, it is enough to show that $f_{n+1}(s,t) = 0$ in $\valun$. 

To prove this 
we use the template method and the transfer principle. 
Let $\Gr^{\C}_{n-1}(\CPn)$ denote the Grassmannian of $\C$-hyperplanes $P \subset \CPn$, equipped with the Haar probability measure $dP$, and define  
 \begin{equation}\label{def bar s}
\bar s := \int_{\Gr^{\C}_{n-1}(\CPn)} \chi(\cdot \, \cap P)\, dP  \in \V^{U(n+1)}_2(\CPn).
\end{equation}
It is not hard to see that the valuations $s,t \in \valun(\C^n)$ correspond via the transfer principle (Theorem \ref{val transfer thm}) to 
$\bar s, t\in \V^{U(n+1)}(\CPn)$ modulo filtration 3 and 2 respectively. Since
$$
\vol(\CPn) = \vol(D^n_{\C})= \frac{\pi^n}{n!} \implies t^{2n}(\CPn) = t^{2n}(D_n^{\C}) = \binom{2n}{n}
$$
by \eqref{powers of t},
it follows from section \ref{val transfer section} that 
\begin{equation}\label{cpn vols}
s ^k t^{2n-2k} (D^n_{\C}) =\bar s ^k t^{2n-2k} (\CPn) =t^{2n-2k} (\C P^{n-k})= \binom {2n-2k}{n-k} , \quad k =0,\dots,n.
\end{equation}

%

The following lemma and its proof were communicated to me by I. Gessel \cite{gessel}. I am informed that it is a special case of the ``Pfaff-Saalsch\"utz identities."
\begin{lemma}
\begin{equation}\label{special pfaff-saalschutz}
\sum_{i=0}^{\lfloor\frac {n+1} 2\rfloor} \frac{(-1)^i }{n+1-i}\binom{n+1-i}{i} \binom{2n-2k-2i}{n-k-i} = \frac {(-1)^{n-k}} {n+1} \binom k{n-k}.
\end{equation}
In particular, the left hand side is zero if $2k<n$.
\end{lemma}
\begin{proof}
In terms of generating functions the identity \eqref{special pfaff-saalschutz} that we desire may be written
\begin{equation}
\sum_{n,k,i}  \frac{(-1)^i (n+1)}{n+1-i}\binom{n+1-i}{i} \binom{2n-2k-2i}{n-k-i}x^ny^k = \frac 1{1-xy(1-x)}.
\end{equation}
But using well-known generating functions \cite{wilf} the left hand side may be expressed as the sum of
\begin{align*}
\sum_{n,k,i}  (-1)^i\binom{n-i}{i-1} \binom{2n-2k-2i}{n-k-i}x^ny^k &= \sum_{m,k,i} (-1)^i \binom{2m} m \binom {m+k}{i-1} x^{m+k+i}y^k\\
&= -\sum_{m,k}  \binom{2m} m (1-x)^{m+k} x^{m+k+1}y^k\\
&= -x\sum_{m}  \binom{2m} m (1-x)^{m} x^m \sum_k (1-x)^kx^ky^y\\
&= \frac{-x}{\sqrt{1-4x(1-x)}(1-xy(1-x))}
\end{align*}
and
\begin{align*}
\sum_{n,k,i}  (-1)^i\binom{n+1-i}{i} \binom{2n-2k-2i}{n-k-i}x^ny^k &=  \sum_{m,k,i} (-1)^i \binom{2m} m \binom {m+k}{i-1} x^{m+k+i}y^k\\
&=\sum_{m,k,i} (-1)^i \binom{2m} m \binom {m+k+1}{i} x^{m+k+i}y^k\\
&= \frac{1-x}{\sqrt{1-4x(1-x)}(1-xy(1-x))},
\end{align*}
hence the sum is 
$$
 \frac{1-2x}{\sqrt{1-4x(1-x)}(1-xy(1-x))} = \frac{1}{1-xy(1-x)}
$$
as claimed.
\end{proof}
With \eqref{cpn vols}, the lemma yields the following identities in $\valun(\C^n)$:
\begin{equation}\label{first relation}
t^{n-2k-1}s^k\cdot\left(\sum_{i=0}^{\lfloor \frac {n+1} 2\rfloor} \frac{(-1)^i }{n+1-i} \binom{n+1-i}{i}t^{n-2i+1} s^{i} \right)
= 0, \quad 0\le 2k<n.
\end{equation}
Since the $t^{n-2k-1}s^k$ in this range constitute a basis of $\Val_{n-1}^{U(n)}(\C^n)$, Alesker-Poincar\'e duality implies that the sum in the second factor is zero. This sum is $(-1)^n f_{n+1}(s,t).$
\end{proof}

Strictly speaking we have now done enough to determine the kinematic operator $k_{U(n)}$ in terms of the first Alesker basis $\{s^kt^l\}$: by the remarks following Thm. \ref{ftaig} the coefficients are the entries of the inverses of Poincar\'e pairing matrices determined by \eqref{cpn vols}. The latter are Hankel matrices with ascending entries of the form $\binom {2k} k$. However, the expressions resulting  from the na\"ive expansion for the inverse seem unreasonably complicated (although perusal of a few of them hints at an elusive structure). A more useful formulation will be given below.

\subsection{Hermitian intrinsic volumes and special cones}

Every element of $\valun(\C^n)$ is a constant coefficient valuation, as follows. Let $(z_1,\ldots,z_n,\zeta_1,\ldots,\zeta_n)$ be canonical coordinates
on $T\mathbb{C}^n\simeq \mathbb{C}^n \times \mathbb{C}^n$, where $z_i=x_i+\sqrt{-1}y_i$ and
$\zeta_i=\xi_i+\sqrt{-1}\eta_i$. The natural action of $U(n)$ on $T\mathbb{C}^n$ corresponds to the
diagonal action on $\mathbb{C}^n \times \mathbb{C}^n$. 
Following Park \cite{pa02}, we consider the elements 
\begin{align*}
\theta_0 & :=\sum_{i=1}^n d\xi_i \wedge d\eta_i \\
\theta_1 & := \sum_{i=1}^n \left(dx_i \wedge d\eta_i-dy_i \wedge d\xi_i\right)\\
\theta_2 & :=\sum_{i=1}^n dx_i \wedge dy_i
\end{align*}
in $\Lambda^2
(\mathbb{C}^n \oplus \mathbb{C}^n)^*$. 
Thus $\theta_2$ is the pullback via the projection map $T\mathbb{C}^n \to \mathbb{C}^n$ of the K\"ahler form of $\mathbb{C}^n$, and $\theta_0 +\theta_ 1 + \theta_2 $ is the pullback of the K\"ahler form under the exponential map $\exp(z,\zeta):= z+ \zeta$. Together with the symplectic form $\omega = \sum_{i =1}^n (dx_i \wedge d\xi_i + dy_i \wedge d\eta_i)$, the $\theta_i$ generate 
the algebra of all $U(n)$-invariant elements in $\Lambda^*
(\mathbb{C}^n \times \mathbb{C}^n)$.

For positive integers $k,q$ with $\max\{0,k-n\} \leq q \leq
\frac{k}{2} \leq n$, we now set  
\begin{displaymath}
\theta_{k,q}:= c_{n,k,q}\theta_0^{n+q-k}\wedge\theta_1^{k-2q}\wedge\theta_2^q
\in \Lambda^{2n}(\C^n \times \C^n)
\end{displaymath}
for $ c_{n,k,q}$ to be specified below, and
put
\begin{equation} \label{eq_mu_integration}
\mu_{k,q}(K):=\int_{N_1(K)} \theta_{k,q}.
\end{equation}

It will be useful to understand the Klain functions associated to the elements of $\valun$. Clearly they are invariant under the action of $U(n)$ on the (real) Grassmannian of $\C^n$.  Tasaki \cite{tasaki00,tasaki03} classified the $U(n)$ orbits of the $\Gr_k(\C^n)$ by defining for $E \in \Gr_k(\C^n)$ the {\bf multiple K\"ahler angle} $\Theta (E) := (0\le \theta_1(E)\le \dots\le\theta_p(E)\le \frac \pi 2), \ p:= \lfloor \frac k 2\rfloor$, via the condition that there exist an orthonormal basis $\alpha_1,\dots,\alpha_k$ of the dual space $E^*$ such that the restriction of the K\"ahler form of $\C^n$ to $E$ is 
$$\sum_{i=1}^{\lfloor \frac k 2 \rfloor} \cos \theta_i \ \alpha_{2i-1}\wedge \alpha_{2i}.$$
This is a complete invariant of the orbit. If $k>n$ then
\begin{displaymath}
\Theta(E)=\bigg(\underbrace{0,\ldots,0}_{k-n},\Theta(E^\perp)\bigg).
\end{displaymath} 
We put $\Gr_{k,q}$ for the orbit of all $E= E^{k,q}$ that may be expressed as the orthogonal direct sum of a $q$-dimensional complex subspace and a complementary $(k-2q)$-dimensional subspace isotropic with respect to the K\"ahler form, i.e.
\begin{displaymath}
\Theta(E^{k,q})=\bigg(\underbrace{0,\ldots,0}_{q},\underbrace{\frac{\pi}{2},\ldots,\frac{\pi}{2}}_{p-q}\bigg).
\end{displaymath}
 In particular, $\Gr_{2p,p}$ is the Grassmannian of $p$-dimensional complex subspaces and $\Gr_{n,0}(\C^n)$ is the Lagrangian Grassmannian.

 It is not hard to see 
\begin{lemma}\label{elem sym basis} The Klain function map gives a linear isomorphism between $\Val_k^{U(n)}$ and the vector space of symmetric polynomials in $\cos^2\Theta(E) := (\cos^2 \theta_1(E),\dots, \cos^2\theta_p(E))$ for $k\le n$ and in $\cos^2\Theta(E^\perp)$ if $k>n$. The hermitian intrinsic volumes are characterized by the condition
\begin{equation}\label{klain mu}
\kl_{\mu_{k,q}}(E^{k,l}) = \delta^q_l
\end{equation}
and the Alesker-Fourier transform acts on them by
\begin{equation}\label{hat mu}
\widehat {\mu_{k,q}} = \mu_{2n-k,n-k+q}.
\end{equation}
\end{lemma}

The valuation $\mu_{n,0} \in \valun $ is called the {\bf Kazarnovskii pseudo-volume}. Note that the restriction of $\mu_{n,0} $ to $\Val^{U(n-1)}(\C^{n-1})$ is zero: there are no isotropic planes of dimension $n$ in $\C^{n-1}$.
Since by Thm. \ref{valun structure}  the kernel of the restriction map  is spanned by the polynomial $f_{n}(s,t)$ we obtain
\begin{lemma} \label{mu = cf} 
For some constant $c$, $\mu_{k,0} = c f_{k}.$
\end{lemma}

The following is elementary.
\begin{lemma}
Put $\Sigma_p$ for the vector space spanned by the elementary symmetric functions in $x_1,\dots,x_p$. 
Let $v_0:= (0,\dots,0), v_1:= (0,\dots,0,1),\dots, v_p:= (1,\dots,1)$ denote the vertices of the simplex $\Delta_p:=0\le x_1\le x_2\le \dots x_p \le 1$ and denote $g_0,\dots,g_{p}$ denote the basis of $\Sigma_p$ determined by the condition $g_i(v_j) = \delta^i_j$. Then the restriction of $g\in \Sigma_p$ to $\Delta_p$ is nonnegative iff $g\in \la g_0,\dots,g_p\ra_+$.
\end{lemma}

From this we deduce at once
\begin{proposition}
The positive cone $P \cap \Val_k^{U(n)} = \la\mu_{k,0},\dots,\mu_{k,p}\ra_+$ .
\end{proposition}

It turns out that the Crofton-positive cone $CP \cap \Val_k^{U(n)}$ is the cone on the valuations $\nu_{k,q}$, where the invariant probability measure on $\Gr_{k,q}$ is a Crofton measure for $\nu_{k,q}$.
Furthermore, using the explicit representation for the hermitian intrinsic volumes in terms of differential forms, it is possible to compute their first variation curvature measures explicitly. Solving the inequalities $\delta \mu_{k,q} \ge 0$ and applying Thm. \ref{monotone characterization}), the monotone cone $M\cap \valun$ has been determined \cite{befu09}. All of these cones are polyhedral. However, in contrast to the $SO(n)$ case described at the end of section \ref{ig of real}, for $n\ge 2$ the three cones are pairwise distinct, and for $n\ge 4$ the monotone cone is not simplicial.

\subsection{Tasaki valuations, and a mysterious duality} \label{tasaki section} 
The {\bf Tasaki valuations} $\tau_{k,q}, q\le \frac k 2$ are defined by the relations
$$
\kl_{\tau_{k,q}} (E) = \sigma_{p,q}(\cos^2\Theta(E))
$$
where $\sigma_{p,q}$ is the $q$th elementary symmetric function in $p:=\lfloor k/2\rfloor$ variables. 

If $k\le n$ then the Tasaki valuations $\tau_{k,q}$ constitute a basis for $\Val_k^{U(n)}$; on the other hand, if $k>n$ then the first K\"ahler angle is identically zero and so the Tasaki valuations are linearly dependent in this range. It is therefore natural to consider the basis of $\valun$ consisting of the $\tau_{k,q}, k \le n$, together with their Fourier transforms (the Fourier transform acts trivially in middle degree $\Val_n^{U(n)}$). We may now write
\begin{equation}
k_{U(n)} (\chi) = \sum_{p,q\le \frac n 2} (T^n_n)_{pq} \tau_{np}\otimes \tau_{nq} + \sum_{k=0}^{n-1} \sum_{p,q\le \min(\frac k 2,  n-\frac k 2)}(T^n_k)_{pq}\, \left(\tau_{k,p}\otimes \widehat{\tau_{k,q}} + \widehat{\tau_{k,q}}\otimes \tau_{k,p} \right)
\end{equation}
where the $T^n_k$ are the {\bf Tasaki matrices}, symmetric matrices of size $\min (\lfloor \frac k 2\rfloor, \lfloor \frac {2n- k} 2\rfloor)$ (note that $\widehat{\tau_{np}}=\tau_{np}$). In these terms the unitary principal kinematic formula exhibits a mysterious additional symmetry.

\begin{theorem}\label{palindrome thm}
If $k=2l\leq n$ is even then 
\begin{equation} \label{eq_palindromic}
(T^n_k)_{i,j}=(T^n_k)_{l-i,l-j}, \quad 0\le i,j\le l.
\end{equation}
\end{theorem}

There are two main ingredients of the proof, arising from two different algebraic representations of the Tasaki valuations and their Fourier transforms. The first has to do with
how the Fourier transform behaves with regard to the representation of the Tasaki valuations in terms of elementary symmetric functions. For simplicity we restrict to valuations of even degree. Let $\Sigma_p$ denote the vector space of elementary symmetric functions $\sigma_{p,q}$ in $p$ variables. As noted above, if $\dim E =2p < n$ then the K\"ahler angle $\Theta(E^\perp)$ is just $\Theta(E) $ preceded by $n-2p$ zeroes, hence $\cos^2\Theta(E^\perp)$ is just $\cos^2\Theta(E)$ preceded by that many ones. In order to express $\widehat{\tau_{2p,q}}$ as a linear combination of Tasaki valuations $\tau_{2n-2p,r}$ in complementary degree, it is enough to express the $q$th elementary symmetric function $\sigma_{p,q}(x_1,\dots,x_p)$ as a linear combination of the $\sigma_{n-p,r}(1,\dots,1,x_1,\dots,x_p)$. However it is simpler to consider the map in the other direction, which corresponds to the map $r:\Sigma_{n-p} \to \Sigma_p$ given by
\begin{equation}
r:f\mapsto \hat f = f(1,\dots,1,x_1,\dots,x_p).
\end{equation}

The second ingredient is the remarkable fact that the Tasaki valuations are monomials under a certain change of variable.

\begin{proposition}\label{tasaki vals} Put $ u:= 4s-t^2.$ Then
\begin{equation} \label{eq_mu_tu}
\tau_{k,q}= \frac{\pi^k}{\omega_k(k-2q)! (2q)!} t^{k-2q}u^q.
\end{equation}
\end{proposition}
It is useful to introduce the formal complex variable $z: = t + \sqrt{-1}v$, where $v$ is formally real and $v^2 =u$. Then
\begin{align*}
\sum_k f_k = \log( 1 + s + t) &= \log \left(1 + t + \frac{t^2} 4 +\frac {v^2} 4\right)\\
& = \log\left(\left| 1 + \frac z 2 \right|^2 \right)\\
& =2\operatorname{Re}\left( \log\left( 1 + \frac z 2 \right) \right)
\end{align*}
whence
\begin{equation}\label{f in z}
f_k = c \re (z^k).
\end{equation}

We introduce the linear involution $\iota:\valun_{2*} \to \valun_{2*}$ on the subspace of valuations of even degree, determined by its action on Tasaki valuations:
$$
\iota(\tau_{2l,q}):=\tau_{2l,l-q}.
$$
Expressing these valuations as real polynomials of even degree in the real and imaginary parts of the formal complex variable $z$, this amounts to interchanging the real and imaginary parts $t$ and $v$, i.e.
$$
\iota(p(z)) = p(\sqrt{-1}\bar z).
$$
This is clearly a formal algebra automorphism of $\R_{even}[t,u]$. Furthermore it descends to an automorphism of the quotient $\valun_{2*}$: to see this it is enough to check that the action of $\iota$ on the basic relations is given by
\begin{align*}
\iota(f_{2k}) &=(-1)^k f_{2k},\\
\iota(tf_{2k-1}) &=(-1)^{k+1}\left(\frac{4k}{2k-1} f_{2k} +  t f_{2k-1}\right).
\end{align*}
This is easily accomplished using the expression \eqref{f in z}.

Furthermore $\iota$ commutes with the Fourier transform. To see this we return to the model spaces $\Sigma_p$ of elementary symmetric functions. Here the map $\iota$ corresponds to $\sigma_{p,q}\mapsto \sigma_{p,p-q}$.
 The assertion thus reduces to the claim that for $m=n-p\ge p$ the diagram
 $$ \begin{CD}
\Sigma_{m} @>{\iota}>> \Sigma_{m}\\
@VV{r}V    @VV{r}V\\
\Sigma_p @>{\iota}>> \Sigma_p
\end{CD}$$
commutes.
 It is enough to prove this for $m= p+1$, in which case $r(\sigma_{p+1,i}) = \sigma_{p,i} + \sigma_{p,i-1}$. Hence for $ i=0,\dots, p+1$,
 $$
 \iota \circ r (\sigma_{p+1,i})  = \iota(\sigma_{p,i} + \sigma_{p,i-1}) = \sigma_{p,{p-i}} +\sigma_{p,p-i+1} = r(\sigma_{p+1,p-i+1} )= r\circ \iota(\sigma_{p+1,i} ).
 $$

\begin{proof}[Proof of Thm. \ref{palindrome thm}] 
Using the additional fact that $\iota$ acts trivially on $\valun_{2n}$, we compute
\begin{align*}
\tau_{2p,i}\cdot \fourier{\tau_{2p,j}}&= \tau_{2p,i}\cdot \fourier{(\iota\tau_{2p,p-j})} \\
&= \tau_{2p,i}\cdot {\iota(\fourier{\tau_{2p,p-j}})}\\
&=\iota\left( \tau_{2p,i}\cdot {\iota(\fourier{\tau_{2p,p-j}})}\right)\\
&= \iota(\tau_{2p,i})\cdot \fourier{\tau_{2p,p-j}}\\
&= \tau_{2p,p-i}\cdot \fourier{\tau_{2p,p-j}}.
\end{align*}
With Theorem \ref{ftaig}, this implies the result.
\end{proof}

\subsection{Determination of the kinematic operator of $(\C^n,U(n))$} The key to carrying out actual computations with the $U(n)$-invariant valuations is the correspondence between their algebraic representations in terms of $t,s,u$ and the more geometric expressions in terms of the hermitian intrinsic volumes and the Tasaki valuations. This is possible due mainly to two facts. The first is 
Lemma \ref{mu = cf}, establishing this correspondence in a crucial special case. The second is the explicit calculation, using Corollary \ref{def lambda}, of the multiplication operator $L$ from Thm. \ref {alesker package} \eqref{hard lefschetz} in terms of the hermitian intrinsic volumes.

 By  part \eqref{thm_sl2_representation} of Proposition \ref{ccv model}, this gives in explicit terms the structure of $\valun$ as an $\sltwo$ module, and
we can apply the well-established theory of such representations (cf. e.g. \cite{grha78}, ). In particular we find the {\it primitive elements} $\pi_{2r,r} \in \valun_{2r}, \ 2r \le n $, characterized by the condition $\Lambda \pi_{2r,r} = L^{2n-2r+1}\pi_{2r,r}= 0$, that generate $\valun $ as an $\frak{sl}(2,\R)$ module in the sense that the valuations
$$
\pi_{2r+k,r}:= cL^k \pi_{2r,r}, \quad 2r\le n, \ \  k \le 2n-2r
$$
constitute another basis of $\valun$. It follows at once from the definition that  the Poincar\'e pairing
$$
(\pi_{kp} , \pi_{lq} )= 0
$$
unless $k+l = 2n, \ p=q$. In other words the Poincar\'e pairing on $\valun$ is {\it antidiagonal} with respect to this basis, and moreover the precise forms of the formulas that led us to this point permit us to calculate the value of the product in the nontrivial cases. Thus Thm. \ref{ftaig} (and the remarks following it) yields the value of $k_{U(n)}(\chi)$ in these terms. 

Unwinding the construction of the primitive basis it is straightforward also to express $k_{U(n)}(\chi)$ in terms of the Tasaki basis, although the coefficients are expressed as singly-indexed sums not in closed form. Since by \eqref{eq_mu_tu} it is easy to compute the product of two Tasaki valuations, the multiplicative property of the kinematic operator allows us now to compute all of the $k_{U(n)}(\tau_{kp})$ and thereby also the $k_{U(n)}(\mu_{kp})$. 

With the transfer principle and \eqref{kg ag} of Thm. \ref{ftaig} we can also compute precise first order formulas in complex space forms, as well as additive kinematic formulas in $\C^n$ (note that the first order formulas in $\CPn$ may be regarded as a generalization of B\'ezout's theorem from elementary algebraic geometry). For example we may
compute the expected value of the length of the curve given by the intersection of a real 4-fold and a real 5-fold in $\C P^4$.
\begin{theorem} Let $M^4, N^5 \subset \C P^4$ be real $C^1$ submanifolds of dimension $4,5$ respectively. Let $\theta_1,\theta_2$ be the K\"ahler angles of the tangent plane to $M$ at a general point $x$ and $\psi$ the K\"ahler angle of the orthogonal complement to the tangent plane to $N$ at $y$. Let $dg$ denote the invariant probability measure on $U(5)$. Then
\begin{align*}
\int_{U(5)}\length &({M \cap g N} )\, dg= \\
\left.\frac{1}{5\pi^4} \times\right[&
 30\vol_4(M)\vol_5(N) 
 -6 \vol_4(M)\int_N \cos^2\psi\, dy \\
&\quad-3\int_M (\cos^2\theta_1+\cos^2\theta_2 )\, dx \cdot \vol_5(N)\\
&\quad\left.+7\int_M (\cos^2\theta_1+\cos^2\theta_2)\, dx\cdot \int_N \cos^2\psi\, dy \right].
\end{align*}
\end{theorem}

The companion result is the following additive kinematic formula for the average 7-dimensional volume of the Minkowski sum of two convex subsets in $\C^4$ of dimensions 3 and 4 respectively. 
\begin{theorem} Let $E\in \Gr_4(\C^4), F\in \Gr_3(\C^4)$; let $\theta_1,\theta_2$ be the K\"ahler angles of $E$ and $\psi$ the K\"ahler angle of $F$. Let $dg$ be the invariant probability measure on $U(4)$. If $A\in \K(E), B\in \K(F)$ then
\begin{align*}\int_{U(4)}\vol_7& (A+ gB)\, dg=\frac{1}{120} \vol_4(A)\vol_3(B) \times \\
&\left[30 -6 \cos^2\psi -3(\cos^2\theta_1+\cos^2\theta_2) + 7\cos^2\psi (\cos^2\theta_1+\cos^2\theta_2)\right].
\end{align*}
\end{theorem}

\subsection{Integral geometry of complex space forms} 
In view of Theorem \ref{val transfer thm}, the first order kinematic formulas are identical in all three complex space forms, so the results above give the first order formulas also for the curved complex space forms $\CPn, \C H^n$. The full kinematic formulas, both for valuations and for curvature measures, have very  recently been worked out in \cite {bfs}. These results require additional ideas and techniques, which we cannot  summarize adequately here. 

However, the quasi-combinatorial approach of section \ref{sect:algebra structure} allows us to express some of these results in purely algebraic form.
The calculations of Gray (specifically Cor. 6.25 and Thm. 7.7 of \cite{gray}) together with the identities \eqref{powers of t} and section 2.5 above, show that the Lipschitz-Killing curvatures of the complex projective space with constant holomorphic sectional curvature $4$ are given by
\begin{equation}\label{lk of cpn}
t^l(\C P^k)=\begin{cases}
\binom{l}{\frac l 2} \binom{k+1}{\frac l 2+ 1} & l \text{ even}\\
0 & l \text{ odd}
\end{cases}
\end{equation}
Thus if we recall the definition \eqref{def bar s} of $\bar s \in \Voneun_2$ 
then
\begin{align}\label{cpn values}
\bar s^k t^l(\CPn)&=\int_{\Gr^{\C}_{n-k}} \chi(\cdot \, \cap Q)\, dQ\\
&= t^l(\C P^{n-k})\\
&=\begin{cases}
\binom{l}{\frac l 2} \binom{n-k+1}{\frac l 2+ 1} & l \text{ even}\\
0 & l \text{ odd}
\end{cases}
\end{align}
By the transfer principle Theorem \ref{val transfer thm}, the monomials in $\bar s, t$ corresponding to the first Alesker basis for $\valun$ constitute a basis of $\Voneun$, and there are no relations in filtrations $\le n$, where we put $\VLun$ for the algebra of invariant valuations on the complex space form of holomorphic sectional curvature $4\lambda$. As in the proof of Thm. \ref{valun structure}, by Poincar\'e duality the ideal $I$ of all relations between $\bar s ,t$ consists precisely of the polynomials $f(\bar s,t)$ such that
$$
( \bar s^k t^l \cdot f(\bar s,t))(\CPn) = 0.
$$
Using the evaluations \eqref{cpn values} this condition amounts to a family of identities among the coefficients of $f$, which are in fact equivalent to the structural results of \cite{bfs}. These identities are very complicated, to the point where it seems implausible to prove them directly in a satisfying way--- even the Wilf-Zeilberger approach of \cite{wz} yields a certificate which is many pages long. On the contrary, at this point it seems better to view the identities as consequences of \cite{bfs}.
%
%

Remarkably, it turns out that for fixed $n$ the algebras $\VLun$ are all isomorphic, and in fact there are a number of interesting explicit isomorphisms among them. One of the most attractive is the following.

\begin{theorem}[Bernig-Fu-Solanes \cite{bfs}] \label{bfs conj} 
$$
\VLun \simeq \R[[s,t]]/\left(f_{n+1}\left(\bar s, \frac{t}{\sqrt{1+\frac{\lambda t^2}{4} }}\right),f_{n+2}\left(\bar s, \frac{t}{\sqrt{1+\frac{\lambda t^2}{4} }}\right)\right)
$$
where the $f_i$ are as in Theorem \ref{valun structure}.
\end{theorem}


Chronologically, the first conjecture about the structure of the algebras $\VLun$ was the following, which was arrived at through a combination of 
luck and consultation with the Online Encyclopedia of Integer Sequences.

\begin{conjecture} \label{cx space form conj}Let $\bar s,t, \lambda$ be variables of formal degrees $2,1, -2$ respectively. Define the formal series $\bar f_k(\bar s,t,\lambda)$ by
\begin{align}
\sum \bar f_k(\bar s, t,\lambda) x^k &=
\log (1 + \bar sx^2 + tx + \lambda x^{-2}+ 3\lambda^2x^{-4} + 13 \lambda^3 x^{-6}+...) \\
&= \log (1 + \bar sx^2 + tx 
\label{eq:crazy expansion}+ \sum \left[\binom{4n+1}{n+1} - 9 \binom{4n+1}{n-1}\right] \lambda^nx^{-2n}).
\end{align}
Then 
$
\bar f_i(\bar s,t,\lambda ) = 0
$
in $\VLun$ for all $i >n$.
\end{conjecture}


Modulo filtration $n+3$, Conjecture \ref{cx space form conj} is independent of $\lambda$ and is therefore true by Thm. \ref{valun structure}. In particular the conjecture holds through dimension $n=2$.
It is possible to reduce the full conjecture to a sequence of identities similar to \eqref{special pfaff-saalschutz}. For example, the validity of the conjecture modulo filtration $n+5$ is equivalent to the family of identities
\begin{align*}
\sum_{j=0}^{\lfloor \frac{n+3} 2\rfloor} (-1)^{j+1}& \frac 1{n-j+4}\binom{n-j+4}{1,j,n-2j+3}\binom{2n-2k-2j}{n-k-j}\\
+ &\sum_{i=0}^{\lfloor \frac{n+1} 2\rfloor} (-1)^i \frac{n-k-i+1}{n-i+1}\binom{n-i+1}{i}\binom{2n-2k-2i-2}{n-k-i-1} = 0,\\
& k=0,\dots \left\lfloor \frac{n-3}{2}\right\rfloor.
\end{align*}
Gessel \cite{gessel} has proved this to be true by a method similar to (but more complicated than)
the proof of \eqref{special pfaff-saalschutz}. Hence we know that the conjecture is true modulo filtration $n+5$, and through dimension $n=4$.

In principle one could continue in this way, but the numerical identities that arise become unreasonably complicated. Nevertheless, using this approach
we have confirmed by direct machine calculation that Conjecture \ref{cx space form conj} is true through dimension $n=16$, which involves the first seven terms of the $\lambda$ series,  and Bernig has reported that it remains true for still larger values of $n$. 

Completely separate from any thoughts of integral geometry, a combinatorial description of the coefficients appearing in \eqref{eq:crazy expansion} has been given by F. Chapoton \cite{chapoton}, who also proved that the $\lambda$ series there defines an algebraic function $g(\lambda)$ satisfying
 $$
 g= f(1-f-f^2), \quad f= \lambda(1+f)^4.
 $$
 This last fact was also observed independently by Gessel.


\section{Concluding remarks} The theory outlined above gives rise to many open problems. Here are a few general directions that should prove fruitful in the next stage of development.

\begin{enumerate}
\item\label{1st prob} {\bf Further exploration of the integral geometry of $(\Cn,U(n))$.} 
The sizable literature that surrounds the classical $SO(n)$ kinematic formula, seemingly out of all proportion with the extreme simplicity of the underlying algebra, suggests that the identities of \cite{befu09} may represent only the tip of the iceberg, and unitary integral geometry may hold many more wonderful surprises even in the relatively well understood euclidean case.
In my opinion it will likely be fruitful simply to play with the algebra and see what comes out.

With luck this approach may lead to better structural understanding of unitary integral geometry--- for example, this is precisely what happened in the development of the duality theorem \ref{palindrome thm}. For instance, is there a better combinatorial model for the kinematic formulas, perhaps something like the devices occurring in Schubert calculus? Does the combinatorial construction of \cite{chapoton} shed light on Conjecture \ref{cx space form conj}?

\item{\bf Structure of the array $\{\VLun\}_{\lambda,n}$.}
%
%
In each dimension $n$ the spaces of curvature measures may be thought of as a single vector space $C= C_n$, equipped with a coproduct $ \tilde k= \tilde k_n:C\to C\otimes C$, independent of curvature. The coalgebra $\VLun$ of invariant valuations in the complex space form of holomorphic sectional  curvature $\lambda$ is then a quotient of $C$ by a subspace $K_\lambda$. Furthermore, $C$ is a module over each $\VLun$. It turns out (\cite{bfs}) that the actions of the various $\VLun$ on $C$ all commute with each other.
%
%

This phenomenon is actually very general, and applies also to the other curvature-dependent families of isotropic spaces, namely the real space forms, the quaternionic space forms and the Cayley space forms. In the real case they are implicit in the discussion of Section \ref{ig of real}. 

\item {\bf Quaternionic integral geometry.} Of the isotropic spaces of whose integral geometry we remain ignorant, the natural isotropic structures on the quaternionic space forms seems most central and mysterious. By a tour de force calculation Bernig \cite{beB} has determined the Betti numbers of the algebras $\valg(\mathbb H^n)$ associated to the quaternionic vector space $\mathbb H^n$ and the natural isotropic actions of $G = Sp(n), Sp(n)\times U(1),Sp(n)\times Sp(1)$. The last group is the appropriate one for the associated family of quaternionic space forms. But 
beyond this nothing is known about the underlying algebra.

Then there are the octonian (Cayley) spaces: the euclidean version is the isotropic action of $Spin(9)$ on $\R^{16}$.

\item {\bf The Weyl principle} gives rise to the Lipschitz-Killing algebra and hence has central importance. Nevertheless it remains poorly understood. For example, the proof above of Thm. \ref{weyl principle} has the form of a numerical accident. 

It appears that the $U(n)$-invariant curvature measures may also restrict on submanifolds to some kind of intrinsic quantities, perhaps invariants of the combination of the induced riemannian metric and the restriction of the K\"ahler form. The same seems to be true of the Holmes-Thompson valuations $\mu_i^F$ of section \ref{ig of rn}, although, curiously, these valuations seem {\it not} to correspond to intrinsic curvature measures.

Furthermore, in conjunction with the analytic questions that we have so far avoided mentioning, the Weyl principle hints at a potentially vast generalization of riemannian geometry to singular spaces. For example, it can be proved by ad hoc means that the Federer curvature measures of a subanalytic subset of $\Rn$ are intrinsic (cf. \cite{fu94,fu00,fu04,be03,bebr03,brkusch97}). We believe that the Weyl principle should remain true for any subspace of $\Rn$ that admits a normal cycle, despite the possible severity of their singularities.

\end{enumerate}

%

\end{document}